\documentclass[%
 aip,
 amsmath,amssymb,
 reprint,%
]{revtex4-1}

\usepackage{graphicx}
\usepackage{dcolumn}
\usepackage{bm}
\usepackage[mathlines]{lineno}

\usepackage{tikz}
\usepackage{amsfonts,amsthm,thmtools}
\usepackage{enumitem}
\usepackage{cleveref}
\usepackage{adjustbox}
\usepackage{overpic}

\newcommand{\panellabel}[1]{\put(-2,70){\textbf{(#1)}}}
\newcommand{\panelcaption}[1]{\par\vspace{-1pt}{\small\centering #1\par}}

\newcommand{\C}{\mathbb{C}}
\newcommand{\R}{\mathbb{R}}

\newcommand{\ww}{\mathbf{w}}
\newcommand{\zz}{\mathbf{z}}
\newcommand{\xx}{\mathbf{x}}

\newtheorem{Proposition}{Proposition}[section]

\theoremstyle{definition}
\newtheorem{ex}{Example}[section]

\makeatletter
\def\@email#1#2{%
 \endgroup
 \patchcmd{\titleblock@produce}
  {\frontmatter@RRAPformat}
  {\frontmatter@RRAPformat{\produce@RRAP{*#1\href{mailto:#2}{#2}}}\frontmatter@RRAPformat}
  {}{}
}%
\makeatother
\begin{document}

\preprint{AIP/123-QED}

\title[Blow-up Landscapes]{Blow-up Parameter Landscapes for Polynomial Dynamical Systems}
\thanks{In honor of Erik Bollt (1967--2025), who had a rare ability to make sophisticated mathematics feel lively and welcoming.  We greatly miss him. }
\author{Emil Graf}
 \email{jeg348@cornell.edu}
 \affiliation{Department of Mathematics, Cornell University, Ithaca, NY 14853, USA}
\author{Ioannis G. Kevrekidis}
\affiliation{Department of Applied Mathematics and Statistics and Department of Chemical and Biomolecular Engineering, Johns Hopkins University, Baltimore, MD 21218, USA}
\author{Alex Townsend}
\affiliation{Department of Mathematics, Cornell University, Ithaca, NY 14853, USA}

\date{\today}

\begin{abstract}
Finite-time blow-up is one of the ways in which a dynamical model can become singular, often signaling the breakdown of either the modeled physical system or the model itself. Determining whether blow-up occurs, and for which parameter values and initial conditions, is therefore a fundamental problem in the analysis of nonlinear dynamical systems. We develop a numerical framework for identifying regions of parameter space in which a dynamical system governed by a system of first-order ordinary differential equations with polynomial right-hand sides exhibits finite-time blow-up for at least one initial condition. The approach combines compactification of the phase space with computational algebraic techniques, producing partitioned parameter landscapes that reveal blow-up and non-blow-up regimes. Through several examples, we show that the method replaces problem-specific hand calculations with an automated computational tool for analyzing blow-up regions in parameter-dependent dynamical systems.
\end{abstract}

\maketitle

\begin{quotation}
Solutions of nonlinear differential equations can become unbounded in finite time.  This behavior is difficult to detect reliably from simulation, especially when the set of initial conditions yielding blow-up is small or when the model depends on parameters.  We introduce a computational-algebraic method for polynomial dynamical systems that divides parameter space into regions over which the blow-up behavior is persistent on each connected parameter region. 
The method compactifies phase space so that blow-up trajectories approach finite limits, then uses algebraic discriminants to find boundaries where the blow-up classification can change. 
The resulting ``blow-up landscapes'' provide a visual and computational way to identify parameter regimes in which finite-time blow-up or unresolved non-directional blow-up can occur.
\end{quotation}

\section{Introduction}
\label{sec:intro}
A trajectory of a dynamical system may remain perfectly regular for a while and then, at a finite time $t^*$, escape every bounded subset of state space, i.e., 
\[
        \|\xx(t)\|\to\infty \qquad \text{as } t\uparrow t^*,
\]
where $\|\cdot\|$ denotes the $2$-norm of a vector.
We say that such a trajectory undergoes a finite-time blow-up. This escape may represent a physical runaway, the onset of a shock, or simply a time horizon beyond which the equations should no longer be trusted.  In applications, it is often of interest how the possibility of blow-up changes as parameters are varied~\cite{matsue2017numerical,matsue2020numerical,moehlis2000bursts}.  

We study parametrized polynomial dynamical systems
\begin{equation} \label{eq:ODE}
        \dot{\xx}=f(\xx,\lambda), \qquad \xx(t)\in\R^d,\quad \lambda\in\R^m,
\end{equation}
where $f$ is a real polynomial in $\xx$ and $\lambda$, of total degree $n$ in $\xx$. We seek a decomposition of parameter space into regions over which the existence of finite-time blow-up is structurally stable.  In our classification a parameter value is marked as admitting blow-up if there is at least one admissible finite initial condition whose trajectory blows up in finite time.  This is a global question about the phase portrait, and it is not reliably answered by direct simulation alone.  A finite collection of initial conditions can easily miss a small basin of escape, while long integrations near a singular transition may confuse very rapid growth with true finite-time blow-up.

Our starting point is the classical idea of compactifying state space, which transforms infinity into a finite boundary and blow-up into trajectories that approach equilibria on that boundary~\cite{elias2006critical,matsue2017numerical,matsue2018blow,matsue2020numerical,kevrekedis2017infinity,bollt2018matching}. Then, for a fixed parameter value, one can study blow-up behavior by studying these boundary equilibria and their stability. The difficulty is that this analysis is usually carried out by hand and individually for each model.  Current approaches do not, by themselves, provide a systematic way to scan a multi-parameter family of systems.

Our main observation is that, for polynomial vector fields, the events that can change the blow-up classification can be defined by polynomial equations (see~\cref{sec:disc}). Projecting the zero sets of these equations to parameter space produces polynomial discriminant boundaries. Away from these boundaries, the relevant compactified phase portrait is structurally stable enough that the existence of finite-time blow-up persists on each region.

This turns the study of finite-time blow-ups into a computational-algebra problem.  We compute discriminants using elimination theory~\cite{cox1997ag}, decompose their complement using numerical algebraic geometry methods~\cite{cummings2025smooth,breiding2025computing,breiding2026elimination}, and then test one representative point in each region. We obtain a partition of parameter space into regions where blow-up is present, absent, or undetermined.

\begin{figure}
\centering
\begin{minipage}[t]{0.48\columnwidth}
\centering
\begin{overpic}[width=\linewidth]{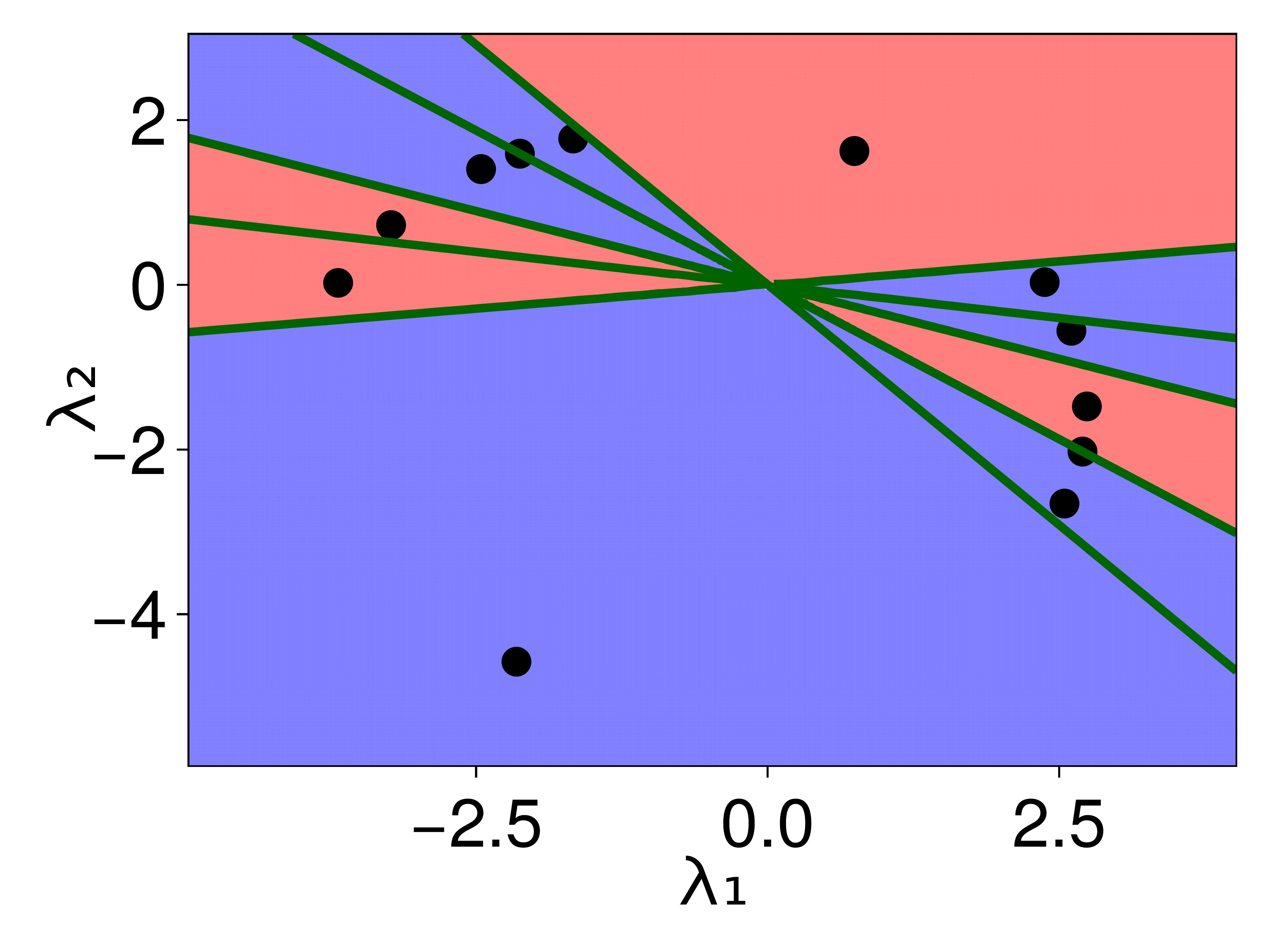}\panellabel{a}\end{overpic}\par
\panelcaption{\adjustbox{scale=0.835}{\footnotesize$\medmuskip=0mu \thickmuskip=.5mu \thinmuskip=0mu\relax\begin{aligned}\dot{x} &= (-\tfrac{1}{12}\lambda_1 -\tfrac{5}{12}\lambda_2)x^2 -\tfrac{1}{3}\lambda_1xy\\ &\quad + (\tfrac{7}{12}\lambda_1 + \tfrac{1}{2}\lambda_2)y^2\\[1pt] \dot{y} &= (\tfrac{1}{2}\lambda_1 + \tfrac{2}{3}\lambda_2)x^2 -\tfrac{1}{3}\lambda_2xy\\ &\quad + (-\lambda_1 + \tfrac{11}{12}\lambda_2)y^2\end{aligned}$}}
\end{minipage}
\hfill
\begin{minipage}[t]{0.48\columnwidth}
\centering
\begin{overpic}[width=\linewidth]{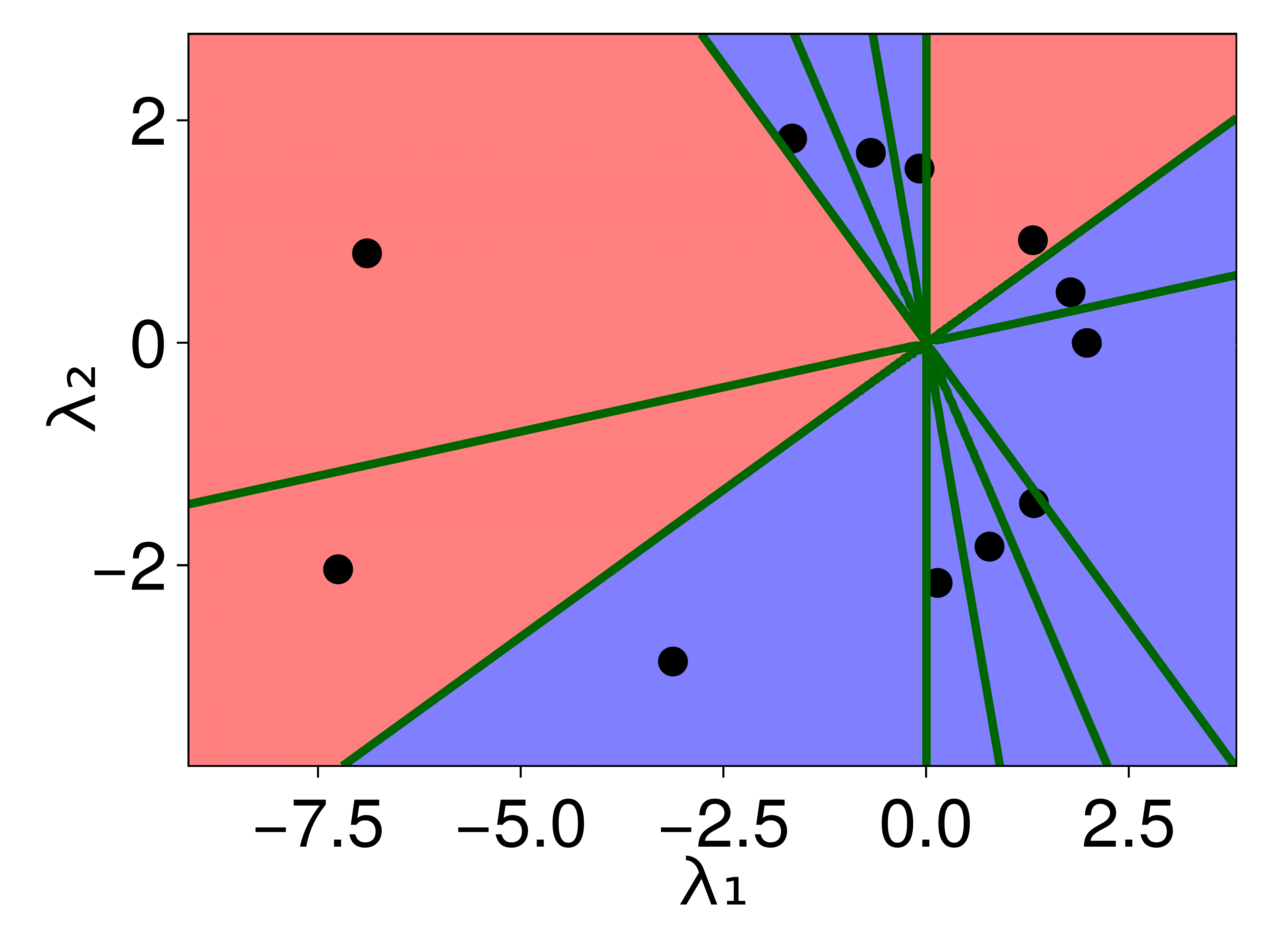}\panellabel{b}\end{overpic}\par
\panelcaption{\adjustbox{scale=0.835}{\footnotesize$\medmuskip=0mu \thickmuskip=.5mu \thinmuskip=0mu\relax\begin{aligned}\dot{x} &= (-\tfrac{1}{4}\lambda_1 + \tfrac{2}{3}\lambda_2)x^2 + (-\tfrac{5}{12}\lambda_1 + \tfrac{1}{3}\lambda_2)xy\\ &\quad + (\tfrac{1}{12}\lambda_1 + \tfrac{1}{12}\lambda_2)y^2\\[1pt] \dot{y} &= \tfrac{1}{3}\lambda_1x^2 + (\tfrac{3}{4}\lambda_1 -\lambda_2)xy\\ &\quad + (-\tfrac{5}{6}\lambda_1 -\tfrac{7}{12}\lambda_2)y^2\end{aligned}$}}
\end{minipage}
\par\vspace{2pt}
\begin{minipage}[t]{0.48\columnwidth}
\centering
\begin{overpic}[width=\linewidth]{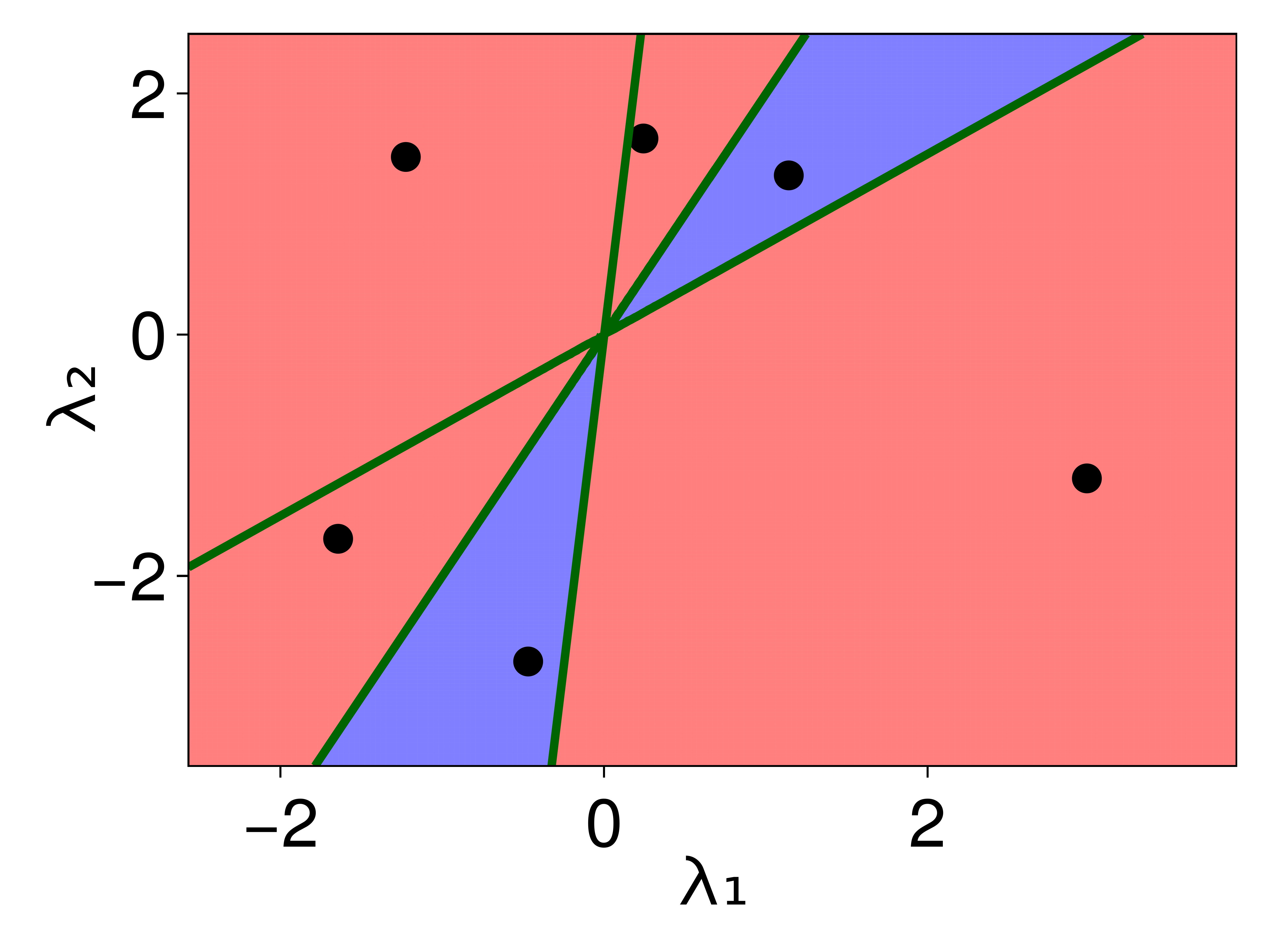}\panellabel{c}\end{overpic}\par
\panelcaption{\adjustbox{scale=0.835}{\footnotesize$\medmuskip=0mu \thickmuskip=.5mu \thinmuskip=0mu\relax\begin{aligned}\dot{x} &= (\tfrac{1}{2}\lambda_1 -\tfrac{2}{3}\lambda_2)x^2\\[1pt] \dot{y} &= (\tfrac{11}{12}\lambda_1 -\tfrac{1}{12}\lambda_2)x^2 + (\tfrac{5}{12}\lambda_1 -\tfrac{1}{4}\lambda_2)xy\\ &\quad + (-\tfrac{1}{3}\lambda_1 + \tfrac{1}{6}\lambda_2)y^2\end{aligned}$}}
\end{minipage}
\hfill
\begin{minipage}[t]{0.48\columnwidth}
\centering
\begin{overpic}[width=\linewidth]{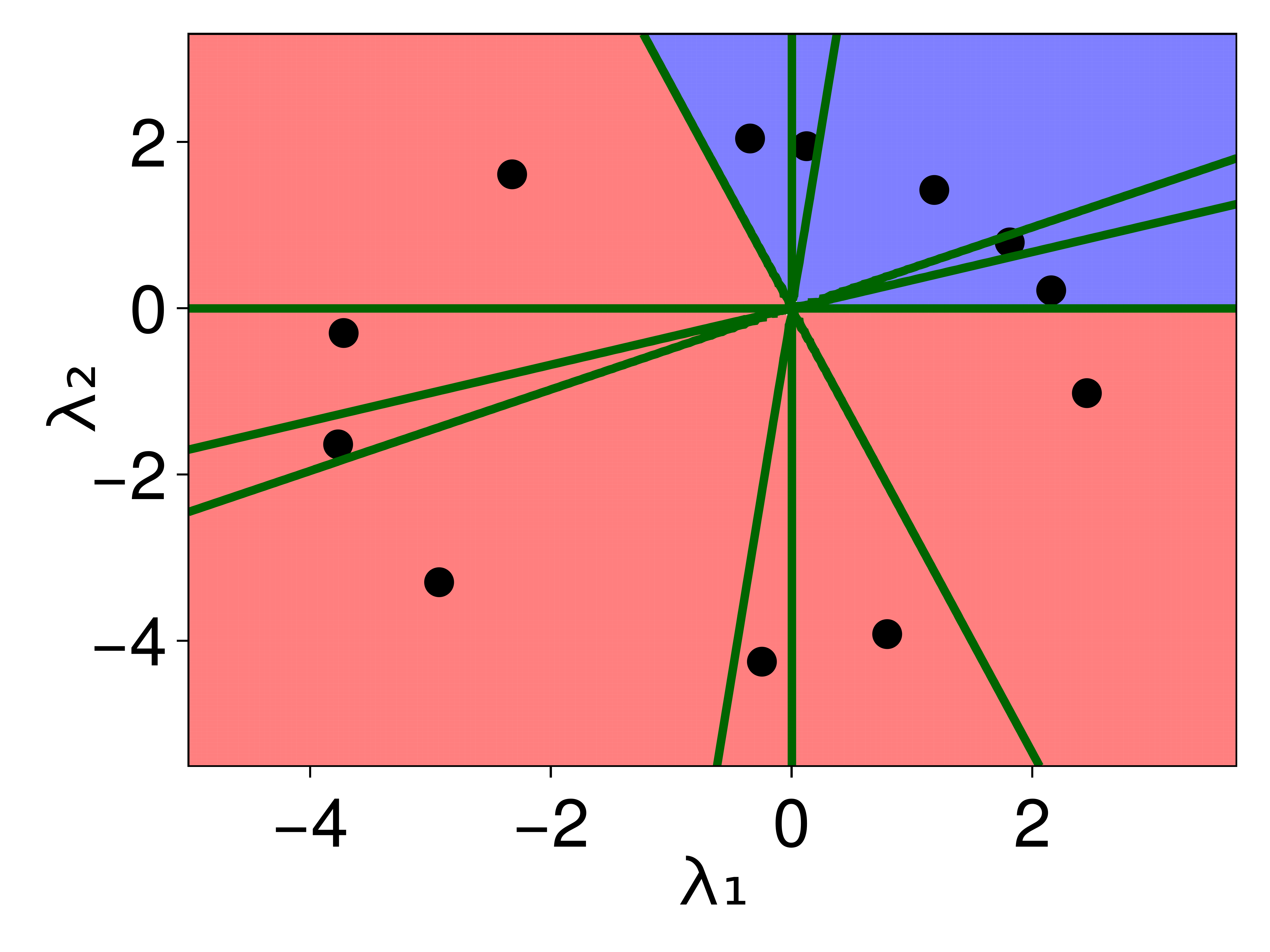}\panellabel{d}\end{overpic}\par
\panelcaption{\adjustbox{scale=0.835}{\footnotesize$\medmuskip=0mu \thickmuskip=.5mu \thinmuskip=0mu\relax\begin{aligned}\dot{x} &= (\tfrac{1}{4}\lambda_1 -\tfrac{7}{12}\lambda_2)x^2 + (\tfrac{7}{12}\lambda_1 + \tfrac{1}{3}\lambda_2)xy\\ &\quad -\tfrac{7}{12}\lambda_1y^2\\[1pt] \dot{y} &= -\tfrac{1}{6}\lambda_2x^2 + (-\lambda_1 + \tfrac{5}{12}\lambda_2)xy\\ &\quad + (-\tfrac{7}{12}\lambda_1 -\tfrac{2}{3}\lambda_2)y^2\end{aligned}$}}
\end{minipage}
\par\vspace{-1pt}
\centerline{\includegraphics[width=\columnwidth,clip,trim=0 44 0 44]{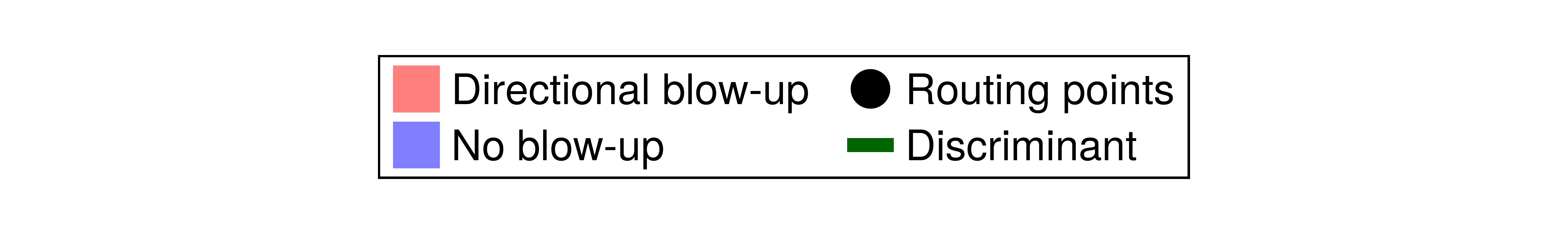}}
\vspace{-15pt}
\caption{Generic two-dimensional quadratic in \cref{subsec:gen}: positive-orthant blow-up landscapes for random linear slices.}
\label{fig:2dquad2}
\end{figure}

\section{Background}
\label{sec:background}

\subsection{Compactification of Blow-up Singularities}
Compactification embeds the $d$-dimensional state space of the system into a compact space, where trajectories that blow up in the original system now approach a finite limit in the new coordinates~\cite{elias2006critical,matsue2017numerical,matsue2018blow,matsue2020numerical,kevrekedis2017infinity,bollt2018matching}. We choose the Poincar\'e compactification, which embeds $\R^d \hookrightarrow S^{d} \subset \R^{d+1}$, the unit sphere in $\R^{d+1}$, by mapping
$$\xx \mapsto \left( \frac{\xx}{\sqrt{1+\|\xx\|^2}},\frac{1}{\sqrt{1+\|\xx\|^2}}\right):= (\zz,y),$$
with $\zz \in \R^d, y \in \R$ and $\|\zz\|^2 + y^2 = 1$. The set of points with $y > 0$ is in bijection with $\R^d$. The points with $y = 0$, which we call the ``sphere at infinity'', do not have a counterpart in the original space. We call equilibria of the compactified system on the sphere at infinity ``equilibria at infinity.'' We also consider the stability, hyperbolicity, and eigenvalues of these equilibria, for the system in \cref{eq:tausys,eq:ytausys}. These are distinct from stability, hyperbolicity, and eigenvalues in the original space, but have a critical role in blow-up; see \cref{sec:discbound} of the Supplementary Material for details. Crucially, $y \to 0$ if and only if $\|\xx\| \to \infty$ in the original space, so the $y$-dynamics alone determine whether a trajectory approaches infinity or not. We call $y$ the ``radial coordinate.''

\subsection{Equilibrium Points at Infinity}
\label{subsec:eqinf}
Under a time rescaling, compactified trajectories corresponding to finite-time blow-up approach the sphere at infinity after infinite rescaled time~\cite{elias2006critical,matsue2017numerical}. We derive a compactified system (\cref{eq:tausys,eq:ytausys}) with rescaled time, which, importantly, remains polynomial.

Let $\kappa = \sqrt{1+\|\xx\|^2}$. We can rewrite \cref{eq:ODE} as
\begin{equation*}
\begin{aligned}
    \frac{d \zz}{dt} &= \frac{d}{dt} (\xx/\kappa) = \kappa^{-1} \frac{d \xx}{d t} - \kappa^{-2} \left\langle \nabla \kappa, \frac{d \xx}{dt} \right\rangle \xx \\
    &= \kappa^{-1} [ f(\xx,\lambda) - \kappa^{-1} \left\langle \zz, f(\xx,\lambda) \right\rangle \xx] \\
    &= \kappa^{-1} [ f(\kappa \zz,\lambda) - \langle \zz, f(\kappa \zz,\lambda) \rangle \zz],
\end{aligned}
\end{equation*}
where $\langle \cdot,\cdot \rangle$ denotes the Euclidean inner product.
The $y$-coordinate determines whether a trajectory approaches infinity in the original space, so we separate the $y$-dynamics:
\begin{equation}
    \frac{dy}{dt}= -\frac{1}{2}(1+\|\xx\|^2)^{-3/2} \frac{d}{dt}\|
    \xx\|^2 = -y^3 \langle \xx, f(\xx,\lambda) \rangle.
\end{equation}
Write
$
f(\xx,\lambda) = f_0(\xx,\lambda) + \cdots + f_n(\xx,\lambda),
$
where $f_i$ consists of the degree $i$ homogeneous terms of $f$ in $\xx$. Define 
\begin{equation} \label{eq:ftilde}
\tilde{f}(\zz,\kappa,\lambda) \! = \! \kappa^{-n} f(\kappa \zz,\lambda)  \! = \! \kappa^{-n} f_0(\zz,\lambda) \! + \! \cdots \! + \! f_n(\zz,\lambda).
\end{equation}
Notice that we can write $\kappa = \kappa(T^{-1}(\zz,\lambda))$, where $T(\xx,\lambda) = \xx/\kappa$. Then we can rewrite the system as
\begin{equation} \label{eq:zsys}
    \frac{d \zz}{dt} = [\kappa(T^{-1}(\zz,\lambda))]^{n-1} [ \tilde{f}(\zz,\lambda) - \langle \nabla \kappa, \tilde{f}(\zz,\lambda) \rangle \zz].
\end{equation}
Then we define the time rescaling by
\begin{equation} \label{eq:taudef}
    \tau = \int_0^t (\kappa(\xx(s)))^{n-1} ds,
\end{equation}
so that
$
    \frac{d\tau}{d t} = \kappa(\xx(t))^{n-1}.
$
The time rescaling turns directional blow-ups into equilibria at infinity, and it depends critically on $n$, the degree of $f$ in $\xx$ in the original system~\cite{elias2006critical,matsue2017numerical}. We can rewrite \cref{eq:zsys} with respect to the rescaled time $\tau$ by
\begin{equation} \label{eq:tausys}
    \frac{d \zz}{d\tau} = \tilde{f}(\zz,\lambda) - \langle \nabla \kappa, \tilde{f}(\zz,\lambda) \rangle \zz := g(\zz,\lambda).
\end{equation}
We can also rewrite the $y$-dynamics as
\begin{equation} \label{eq:ytausys}
    \frac{d y}{d \tau} = -y^{n+2} \langle \xx, f(\xx,\lambda) \rangle.
\end{equation}
The reason to perform this transformation is to convert blow-ups into equilibria.  We say that a solution $\xx(t)$ has a directional finite-time blow-up in the direction $\xx^*$ if $\lim_{t \to t^*} \|\xx(t)\| = \pm \infty$ and $\lim_{t \to t^*} \frac{\xx(t)}{\|\xx(t)\|} = \xx^*$. The following result states that every directional blow-up gives an equilibrium point at infinity.
\begin{Proposition}[Ref.~\onlinecite{elias2006critical},~Prop.~2.5]
\label{prop:infconv}
Assume that a solution $\xx(t)$ to the ODE in \cref{eq:ODE} has a maximal time-interval of existence $(a,b) \subset \R$, and that $\xx$ tends to infinity in the direction $\xx^*$. Then $\xx^*$ is an equilibrium point of \cref{eq:tausys} with $y = 0$, where $y = \frac{1}{\sqrt{1+\|\xx\|^2}}$.
\end{Proposition}
It is possible to have $b = \infty$. Such systems, for example $\dot{x} = x$, with solution $Ce^{t}$, do not have finite-time blow-up, and are instead referred to as \emph{grow-up} trajectories~\cite{matsue2017numerical}. We can distinguish between blow-up and grow-up trajectories using the integral~\cite{matsue2017numerical}
\begin{equation} \label{eq:tmax}
    t_{\textrm{max}} = \int_{0}^{\infty} \frac{d \tau}{(\sqrt{1 + \|T^{-1}(\zz(\tau))\|^2})^{n-1}}.
\end{equation}
There is a finite-time blow-up if and only if $t_{\textrm{max}} < \infty$.

\subsection{Algebraic Elimination and Parameter Space Decomposition}

A set of equations $p_1,\ldots,p_k$ in the polynomial ring $\R[\xx,\lambda]$ defines an algebraic set $\mathcal{V}(p_1,\ldots,p_k) = \{(\xx,\lambda) : p_i(\xx,\lambda) = 0, 1 \leq i \leq k \}$. There is a natural projection $\pi_{\lambda}:\R^{d+m} \to \R^m$ given by $\pi_{\lambda}(\xx,\lambda) = \lambda$. Algebraic elimination computes a defining set of polynomials (an ideal) for the algebraic set $\pi_{\lambda}(\mathcal{V}(p_1,\ldots,p_k))$.~\cite{cox1997ag} This set can be computed symbolically using classical algebraic geometry methods. Algebraic elimination works over $\C$, but we can always project the resulting algebraic set back down to $\R$, which is standard~\cite{cummings2025routing}.

Previous work finds algebraic sets in $\R^{d+m}$ whose projection separates regions that have distinct stable equilibria~\cite{cummings2025routing}. We apply this approach to blow-ups. We list a set of polynomials in $\R[\xx,\lambda]$ such that the projection of their zero set splits the parameter space into regions on which the blow-up behavior is structurally stable. Projecting these polynomials gives a defining set for the boundaries of all the blow-up regions in the parameter space.

Decomposing the parameter space into the regions defined by this set is challenging. The classical symbolic method is \emph{cylindrical algebraic decomposition}, which is extremely expensive.~\cite{basu2006algorithms}
Fortunately, a highly useful numerical alternative has emerged in refs.~\onlinecite{cummings2025smooth,breiding2025computing,breiding2026elimination}. These numerical methods compute a skeletonized decomposition of the parameter space using so-called routing functions~\cite{cummings2025smooth}.

The zero set of a real polynomial $p \in \R[x_1,\ldots,x_d]$ splits $\R^d$ into finitely many connected components. A routing function for $p$ is a function $h:\R^d \to \R$ that has at least one critical point in each of these connected components, which can be explicitly written~\cite{cummings2025smooth,breiding2025computing,breiding2026elimination} down given $p$. The algorithms of refs.~\onlinecite{cummings2025smooth,breiding2025computing,breiding2026elimination} compute the critical points of the routing function, also known as the routing points, using numerical homotopy continuation. The critical points lying in the same region can be connected by gradient flow, which can also be used for membership testing. The full framework is implemented in Julia~\cite{breiding2025computing,breiding2026elimination}.

The decomposition gives at least one sample point in each region, and then allows membership testing of arbitrary points~\cite{breiding2025computing}. For our purposes, this means we can test whether or not there is blow-up at finitely many sample points, and then obtain an extremely simple test for any other point in the parameter space by running the regional membership test~\cite{breiding2025computing}.

\section{The Discriminant Boundary Between Blow-up Regions}
\label{sec:disc}

To set up our decomposition, we need to enumerate the ways in which an equilibrium at infinity can form or disappear as we vary the parameters, and the ways in which the convergence of the integral in \cref{eq:tmax} can change. The boundaries are the following, which are explained in detail in \cref{sec:discbound} of the Supplementary Material:
\begin{enumerate}[noitemsep,leftmargin=*]
    \item ``The Algebraic Discriminant'' $X_A$: The creation or destruction of an equilibrium at infinity.
    \item ``The Radial Discriminant'' $X_R$: The change of an equilibrium at infinity between attracting and repelling in the $y$-direction.
    \item ``The Stability Discriminant'' $X_S$: An eigenvalue of the Jacobian of $g$ in \cref{eq:tausys} at an equilibrium at infinity changing from positive real part to negative real part.
    \item ``The Degree Discriminant'' $X_D$: The change in degree of $f$ in $\xx$.
    \item ``The Positive Discriminant'' $X_P$: The creation or destruction of positive equilibria at infinity.
\end{enumerate}
The sets $X_S$ and $X_D$ are not regional boundaries, but instead function as exceptional sets where our analysis breaks down, so we always suppose that $\lambda \in \R^m \setminus (X_S \cup X_D)$. In the regions of parameter space defined by $X_A \cup X_R$, the following are constant:
\begin{enumerate}[noitemsep,leftmargin=*]
    \item The number of equilibria at infinity and the sign of the radial eigenvalue $\mu_\perp = \frac{\partial}{\partial y} \left( \frac{dy}{d \tau} \right)$ of each equilibrium.
    \item The degree of $f$, which is always equal to $n$ on the complement of the degree discriminant.
\end{enumerate}
The following proposition characterizes precisely when there is a blow-up. Note that an equilibrium point of the system in \cref{eq:tausys} is called hyperbolic if the eigenvalues of the Jacobian of $g$ all have nonzero real part.
\begin{Proposition} \label{prop:decomp}
    Consider the ODE in \cref{eq:ODE}, with $\xx \in \R^d$. Suppose $\lambda$ is fixed at a parameter value for which the degree of $f(\xx,\lambda)$ in $\xx$ is $n$, and all equilibria at infinity are hyperbolic. Suppose $n > 1$: If there is an equilibrium at infinity with $\mu_\perp < 0$, then there is an initial condition for which the system $\dot{\xx} = f(\xx,\lambda)$ blows up in finite time. Suppose $d \leq 2 $. If there exists at least one equilibrium at infinity, but there are no equilibria with $\mu_\perp < 0$, then the system does not have a finite-time blow-up for any initial condition.
    If $n = 1$, then there is no finite-time blow-up for any initial condition.
\end{Proposition}

For the proof, see Supplementary Material, \cref{subsec:decompproof}.
We can now sort between parameter values where the system:
\begin{enumerate}[noitemsep,leftmargin=*]
\item has a directional finite-time blow-up (convergence to a limit point at $\infty$) for some finite initial condition,
\item cannot have a finite-time blow-up of any kind for any initial condition, and
\item cannot have directional finite-time blow-up, but may have other types of blow-up.
\end{enumerate}

We say that a trajectory blows up ``non-directionally'' if it blows up in finite time but $\xx(t)/\|\xx(t)\|$ does not tend to a finite limit~\cite{dumortier1999polynomial,matsue2025blow}. For example, if the limit set of the compactified system is a periodic orbit, then the system exhibits ``spiral blow-up''; compactified trajectories spiral out towards the sphere at infinity. We still observe many applications to which our framework applies fully. When $d \geq 3$ or there are no equilibria at infinity, we sort between directional blow-up and no directional blow-up, leaving other types of blow-up unresolved.

If only positive-orthant blow-up, where $x(t) \in (0,\infty)^d$ for all $t$ is of interest, the positive discriminant $X_P$ separates regions where the equilibria at infinity may have different signs on the coordinates. The previous theorem applies with positive-orthant blow-up replacing blow-up if we include the positive discriminant. We use this for most physical applications.

\section{Applications} \label{sec:ex}

For each example, we calculate all the above discriminants using symbolic elimination in Macaulay2~\cite{M2,cox1997ag}. Then we decompose the parameter space using the numerical methods of refs.~\onlinecite{cummings2025smooth,breiding2025computing}. We test each sample point obtained by the decomposition algorithm to determine whether or not the system has finite-time blow-up in a particular region by testing all equilibria at infinity to see if they are radially attracting. For all examples we take slices with $m=2$ parameters for visualization and typically keep the state space dimension $d \leq 4$.

Some examples we consider are large enough that either the symbolic elimination step or the numeric regional decomposition step does not terminate in a reasonable amount of time. If our symbolic-numeric pipeline is computationally infeasible, then we sample the discriminant sets numerically using a pseudo-witness set, a well-chosen linear slice that fully captures the projected discriminant set~\cite{hauenstein2010witness}. We then test a sample grid for blow-up, refining using the philosophy of a computer graphics algorithm called marching squares~\cite{maple2003ms} (see \cref{subsec:gen}).

\subsection{Generalized Lotka--Volterra}
\label{subsec:lv}

A generalized Lotka--Volterra (gLV) system $\dot x_i = x_i\big(r_i + \sum_j A_{ij}x_j\big)$ is a canonical model for interacting non-negative quantities, employed in a wide variety of applications~\cite{kusbeyzi2011lv,may1975nonlinear,jimenez2025dynamic,venturelli2018deciphering,hening2018persistence,stein2013ecological,vano2006chaos}. In refs.~\onlinecite{liang2024global,liang2025global}, the authors apply the Poincar\'e compactification to several Lotka--Volterra models with restricted coefficients to obtain global classifications of the dependence of equilibria, both finite and infinite, on parameters. We can accomplish a similar classification extremely efficiently using our framework.

As usual, the discriminant sees only the leading quadratic terms, i.e., $\dot x_i = x_i\sum_j A_{ij}x_j$. Each panel of \cref{fig:kolmogorov} is the positive-orthant parameter landscape over two swept interaction coefficients of a particular generalized Lotka--Volterra model. Our selected models are:
\begin{enumerate}[label=\textbf{(\alph*)},noitemsep,leftmargin=*]
  \item Cyclic competition among three species~\cite{may1975nonlinear}.
  \item Competitive Lotka--Volterra tumor-host model~\cite{jimenez2025dynamic}.
  \item Microbial interactions in synthetic human gut microbiome communities, pairwise model with quadratic self-interaction coefficients set to $-1$~\cite{venturelli2018deciphering}.
  \item Slice of a three species food chain with quadratic self-interaction coefficients set to $-1$~\cite{hening2018persistence}.
  \item Model of gut microbiota, restricted to the subsystem governing the interaction of \emph{Enterococcus},
\emph{Blautia}, and \emph{C. difficile}. The fixed coefficients are experimental values~\cite{stein2013ecological}.
  \item Four-dimensional gLV model~\cite{vano2006chaos}.
\end{enumerate}
We compute panels (a)--(e) by symbolic elimination and numeric decomposition, and (f) by numeric sampling.

\begin{figure}
\centering
\renewcommand{\panelcaption}[1]{\par\vspace{-4pt}{\small\centering #1\par}}
\begin{minipage}[t]{0.48\columnwidth}\centering
\begin{overpic}[width=\linewidth]{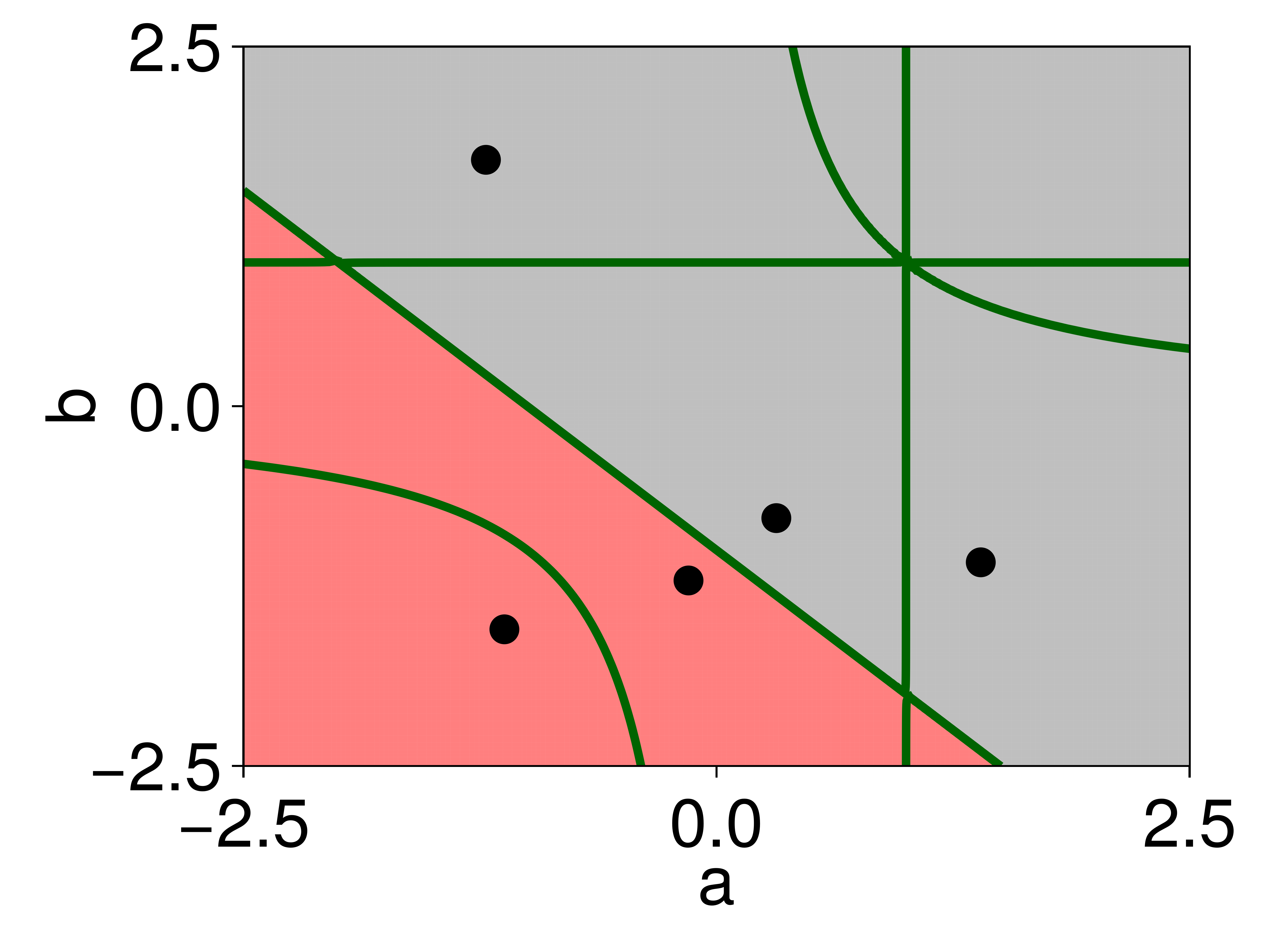}\panellabel{a}\end{overpic}\par
\panelcaption{\adjustbox{max width=\linewidth}{\footnotesize$\dot{x}_i = x_i\big(1 - x_i - a\,x_{i+1} - b\,x_{i-1}\big)$}}
\end{minipage}\hfill
\hfill
\begin{minipage}[t]{0.48\columnwidth}\centering
\begin{overpic}[width=\linewidth]{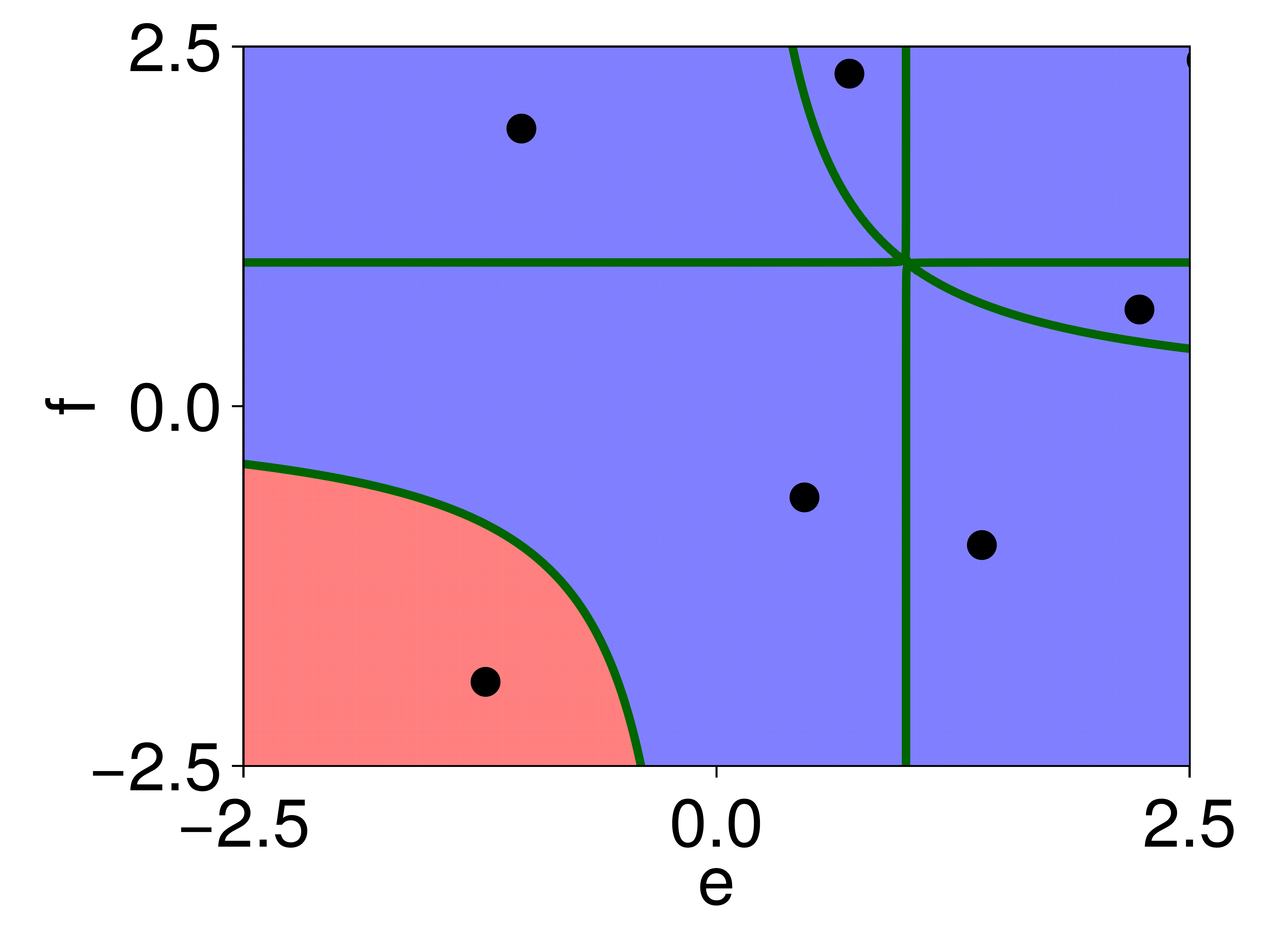}\panellabel{b}\end{overpic}\par
\panelcaption{\adjustbox{max width=\linewidth}{\footnotesize$\medmuskip=0mu \thickmuskip=.5mu \thinmuskip=0mu\relax\begin{aligned}\dot{u} &= u(1 - u - e\,v)\\[-2pt] \dot{v} &= v(d - v - f\,u)\end{aligned}$}}
\end{minipage} \hfill
\par\vspace{2pt}
\begin{minipage}[t]{0.48\columnwidth}\centering
\begin{overpic}[width=\linewidth]{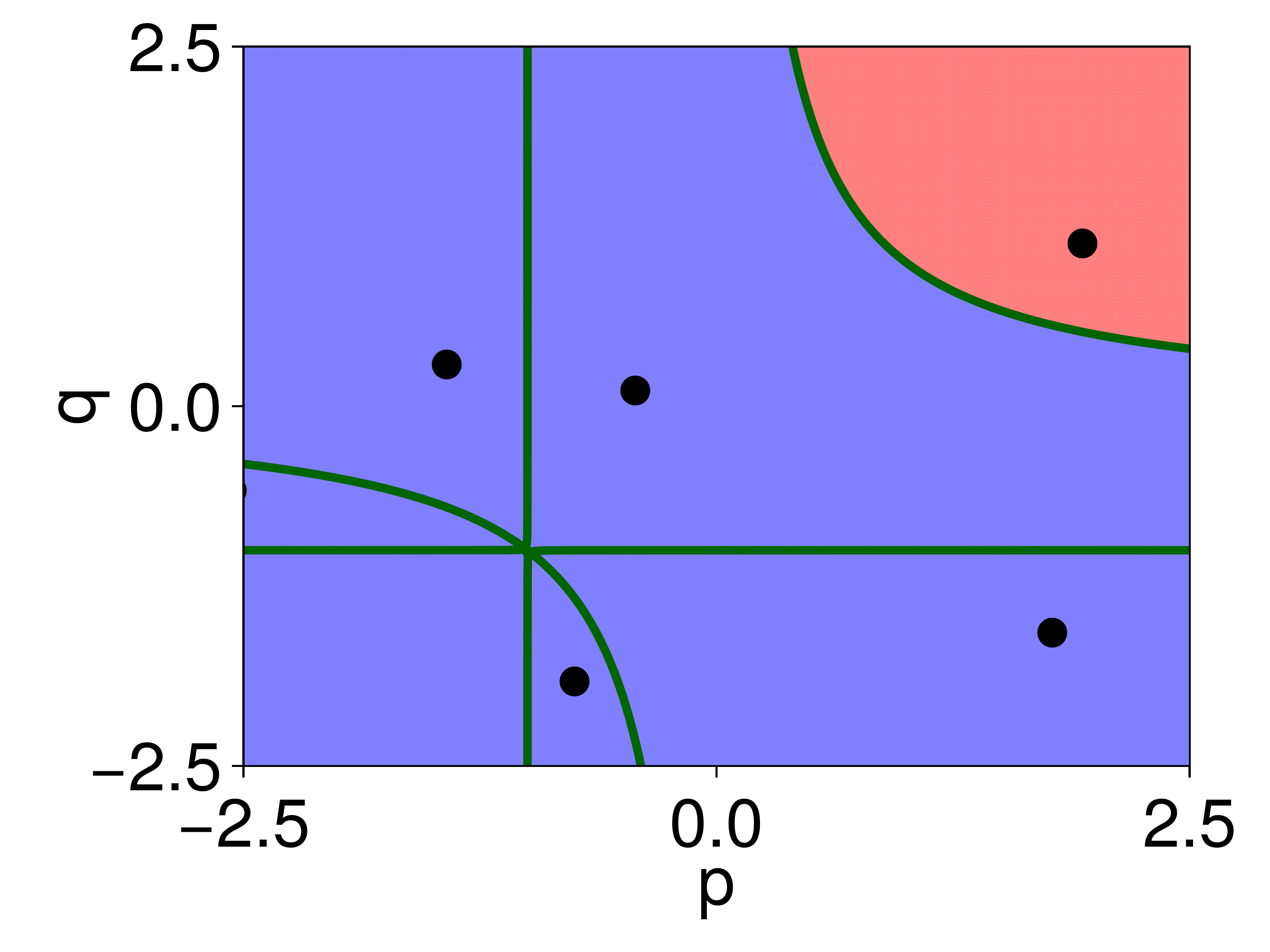}\panellabel{c}\end{overpic}\par
\panelcaption{\adjustbox{max width=\linewidth}{\footnotesize$\medmuskip=0mu \thickmuskip=.5mu \thinmuskip=0mu\relax\begin{aligned}\dot{x}_1 &= x_1(\mu_1 - x_1 + p\,x_2)\\[-2pt] \dot{x}_2 &= x_2(\mu_2 + q\,x_1 - x_2)\end{aligned}$}}
\end{minipage}
\hfill
\begin{minipage}[t]{0.48\columnwidth}\centering
\begin{overpic}[width=\linewidth]{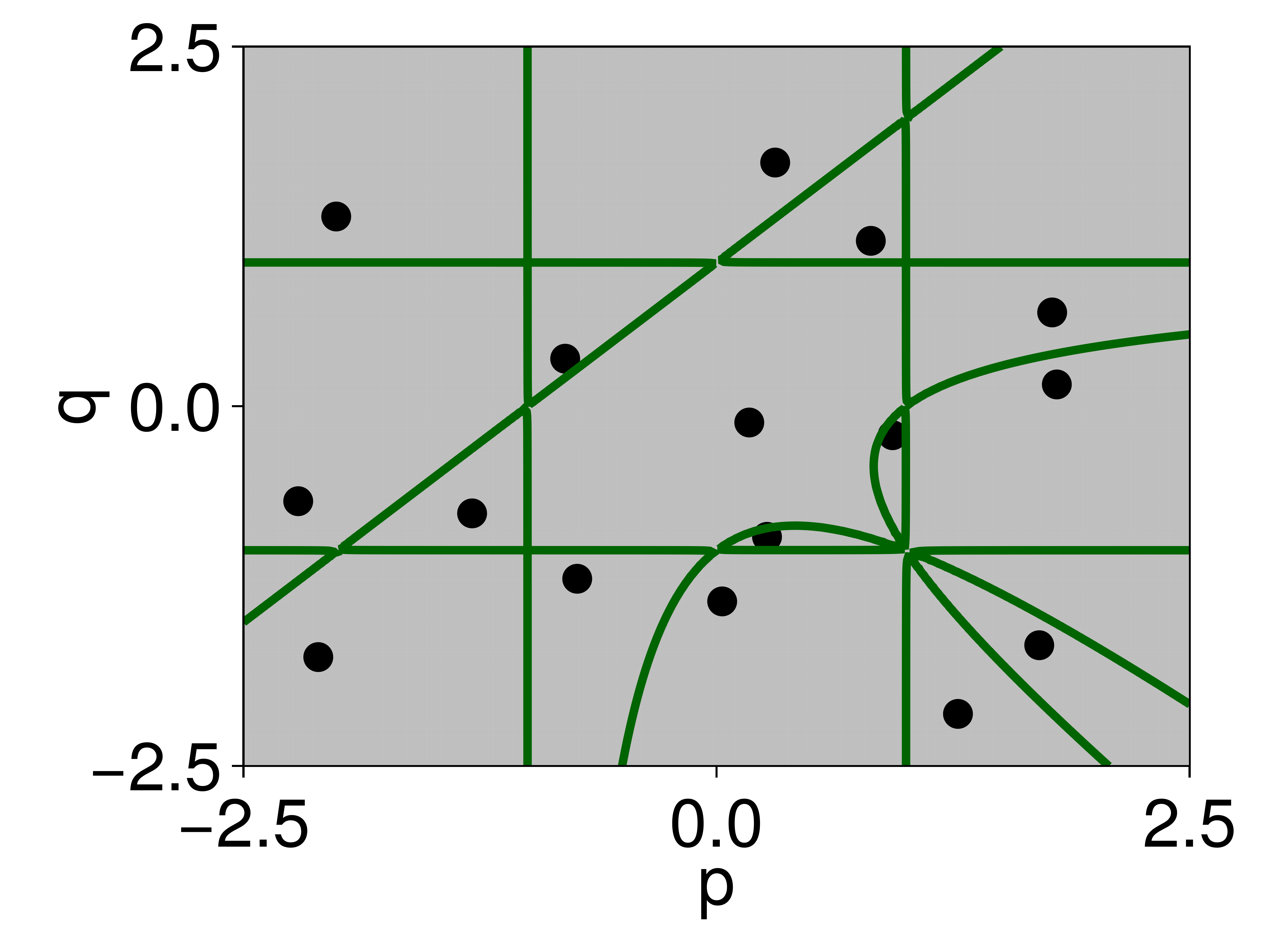}\panellabel{d}\end{overpic}\par
\panelcaption{\adjustbox{max width=\linewidth}{\footnotesize$\medmuskip=0mu \thickmuskip=.5mu \thinmuskip=0mu\relax\begin{aligned}\dot{x}_1 &= x_1(r_1 - x_1 - p\,x_2)\\[-2pt] \dot{x}_2 &= x_2(-r_2 + p\,x_1 - x_2 - q\,x_3)\\[-2pt] \dot{x}_3 &= x_3(-r_3 + q\,x_2 - x_3)\end{aligned}$}}
\end{minipage}\hfill
\par\vspace{2pt}
\begin{minipage}[t]{0.48\columnwidth}\centering
\begin{overpic}[width=\linewidth]{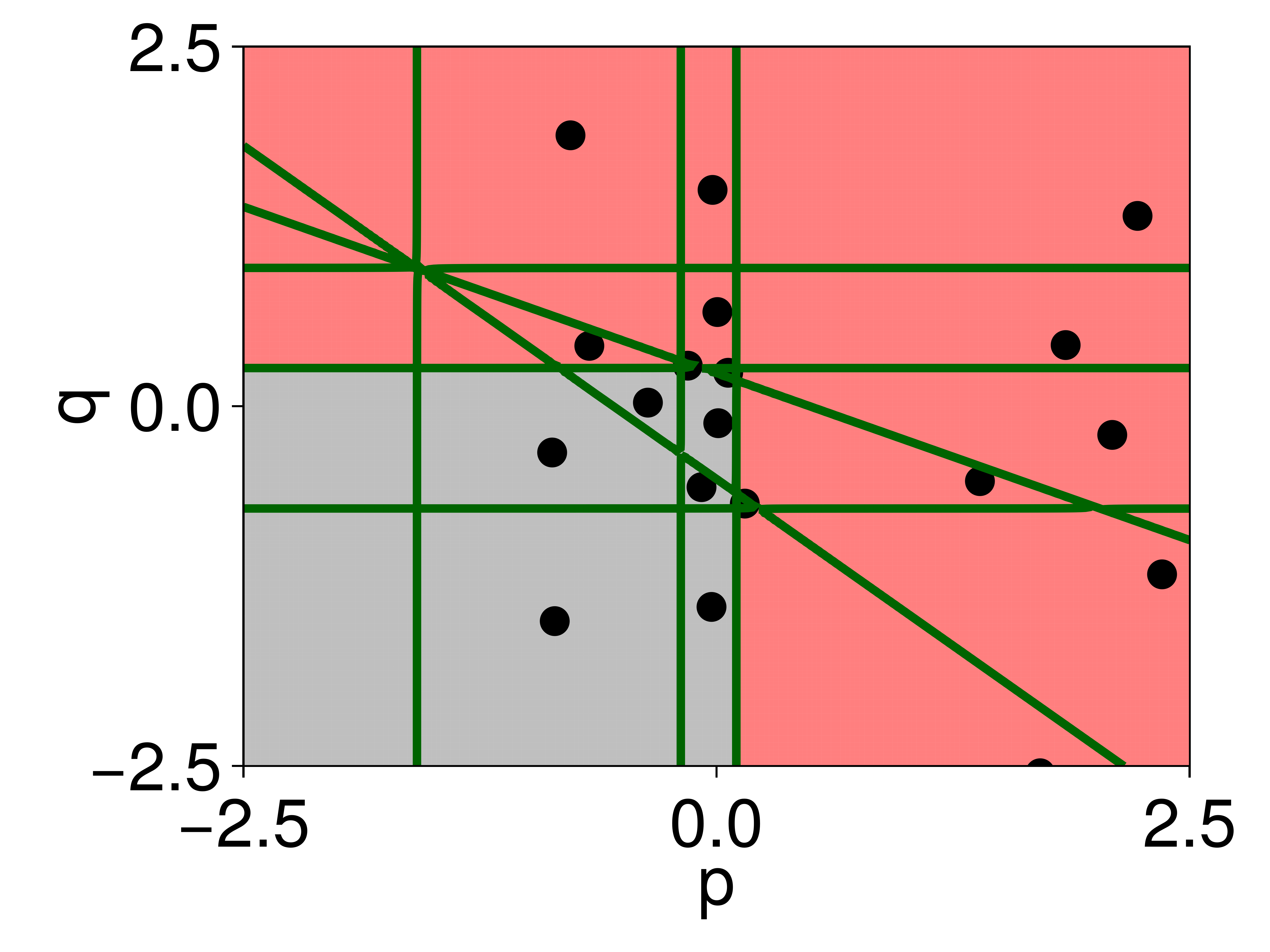}\panellabel{e}\end{overpic}\par
\panelcaption{\adjustbox{max width=\linewidth}{\footnotesize$\medmuskip=0mu \thickmuskip=.5mu \thinmuskip=0mu\relax\begin{aligned}\dot{x}_1 &= x_1(\mu_1 - 0.19x_1 - 0.33x_2 + 0.11x_3)\\[-2pt] \dot{x}_2 &= x_2(\mu_2 + 0.22x_1 - 0.71x_2 + 0.16x_3)\\[-2pt] \dot{x}_3 &= x_3(\mu_3 + p\,x_1 + q\,x_2 - 0.06x_3)\end{aligned}$}}
\end{minipage}
\hfill
\begin{minipage}[t]{0.48\columnwidth}\centering
\begin{overpic}[width=\linewidth]{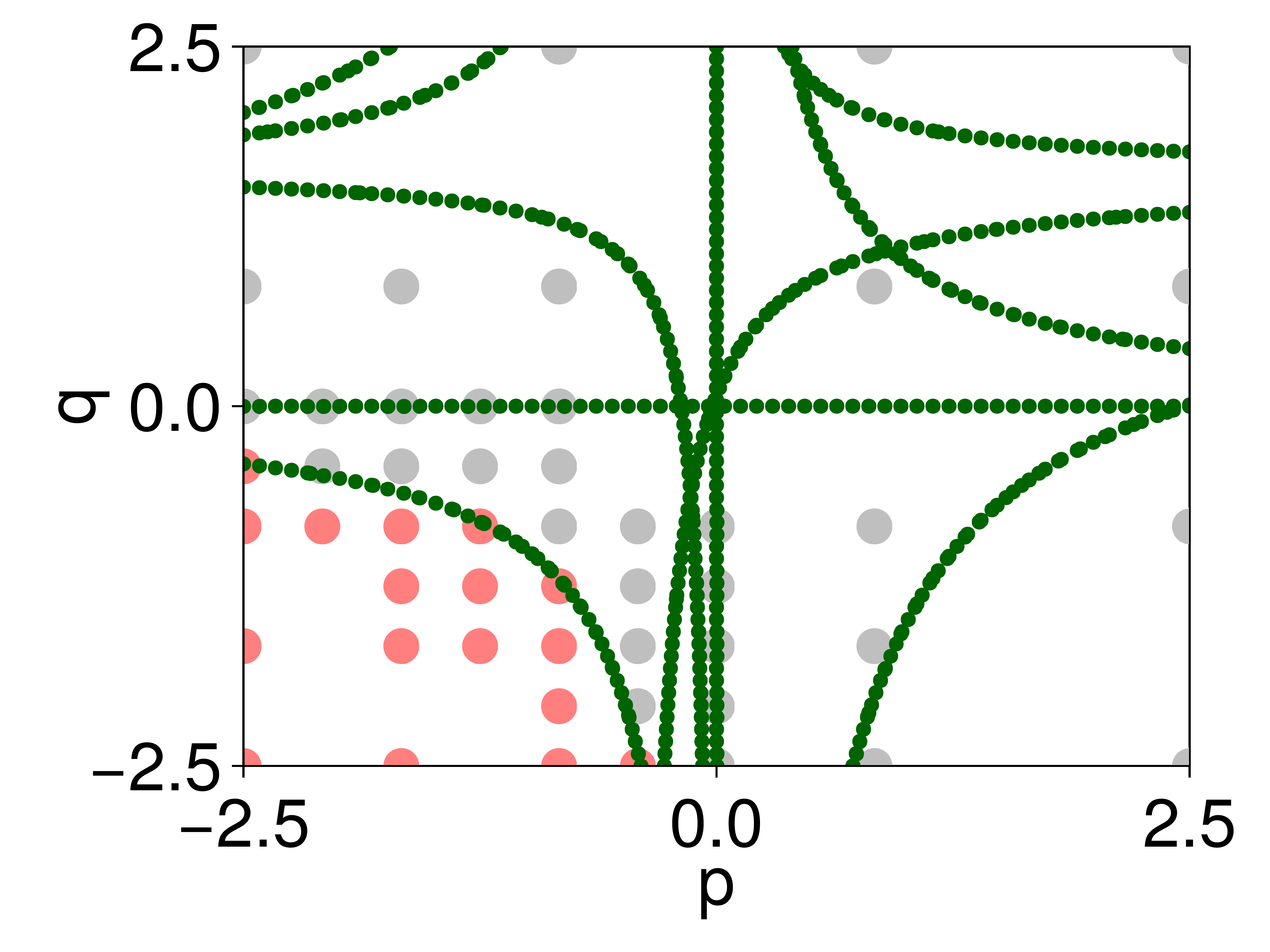}\panellabel{f}\end{overpic}\par
\panelcaption{\adjustbox{max width=\linewidth}{\footnotesize$\begin{gathered}\dot{x}_i = r_i\,x_i\big(1 - \textstyle\sum_j A_{ij}x_j\big),\\ (p,q)=(A_{12},A_{21})\end{gathered}$}}
\end{minipage}
\par\vspace{-1pt}
\centerline{\includegraphics[width=\columnwidth,clip,trim=0 75 0 76]{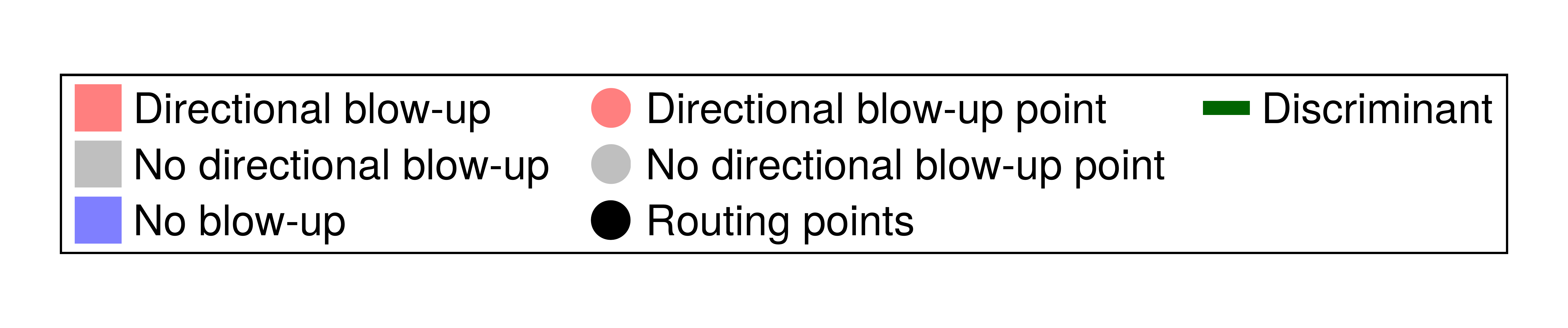}}
\vspace{-15pt}
\caption{Positive-orthant blow-up landscapes for six generalized Lotka--Volterra models.}
\label{fig:kolmogorov}
\end{figure}

\subsection{Generic Vector Fields}
\label{subsec:gen}

Our framework is most powerful when applied to large, generic parametrized families in \cref{eq:ODE}. Such families can arise, for example, when considering traveling wave ODEs associated with shock formation for conservation law systems~\cite{schecter2002traveling,achleitner2012saddle,keyfitz2012conserving,keyfitz1995spaces,kranzer1990strictly,keyfitz2011singular,matsue2020numerical,matsue2018blow,matsue2025blow}, where there is interest in the family of planar quadratic systems~\cite{isaacson1990transitional,schaeffer1987classification}. Ref.~\onlinecite{li2021survey} includes a broad survey of work on planar quadratic systems, including significant analysis of the compactified phase portraits of such systems; our work provides the computational algebraic counterpart to the analysis previously done by hand. Our computational tool also allows us to quickly push beyond planar quadratic systems.

Only the highest degree terms of $F$ are relevant for blow-up, so we probe systems of the form
\begin{equation} \label{eq:gen}
\dot{u} = F(u),
\end{equation}
where $F$ is a bivariate and trivariate homogeneous quadratic function. For visualization, we probe random linear slices of their coefficient space. We use positive-orthant blow-up for all the examples.

In \cref{fig:2dquad2}, we probe quadratic two-dimensional systems. Each leading coefficient $c$ is set to a linear form $c = A\,\lambda_1 + B\,\lambda_2$ with two free parameters, $A,B$ having numerator and denominator uniformly sampled from $[-12,12] \setminus \{0\}$. 

In \cref{fig:gen3d}, we probe quadratic three-dimensional systems. The standard pipeline is too expensive to compute, so we sample the discriminant boundaries using a pseudo-witness set~\cite{hauenstein2010witness}. We test sample points for blow-up, using marching squares to refine our grid.
As $d = 3$, we can not mark regions as ``no blow-up''; we can only say that there is no directional blow-up, with other types of blow-up undetermined.

\begin{figure}
\centering
\begin{minipage}[t]{0.48\columnwidth}
\centering
\begin{overpic}[width=\linewidth]{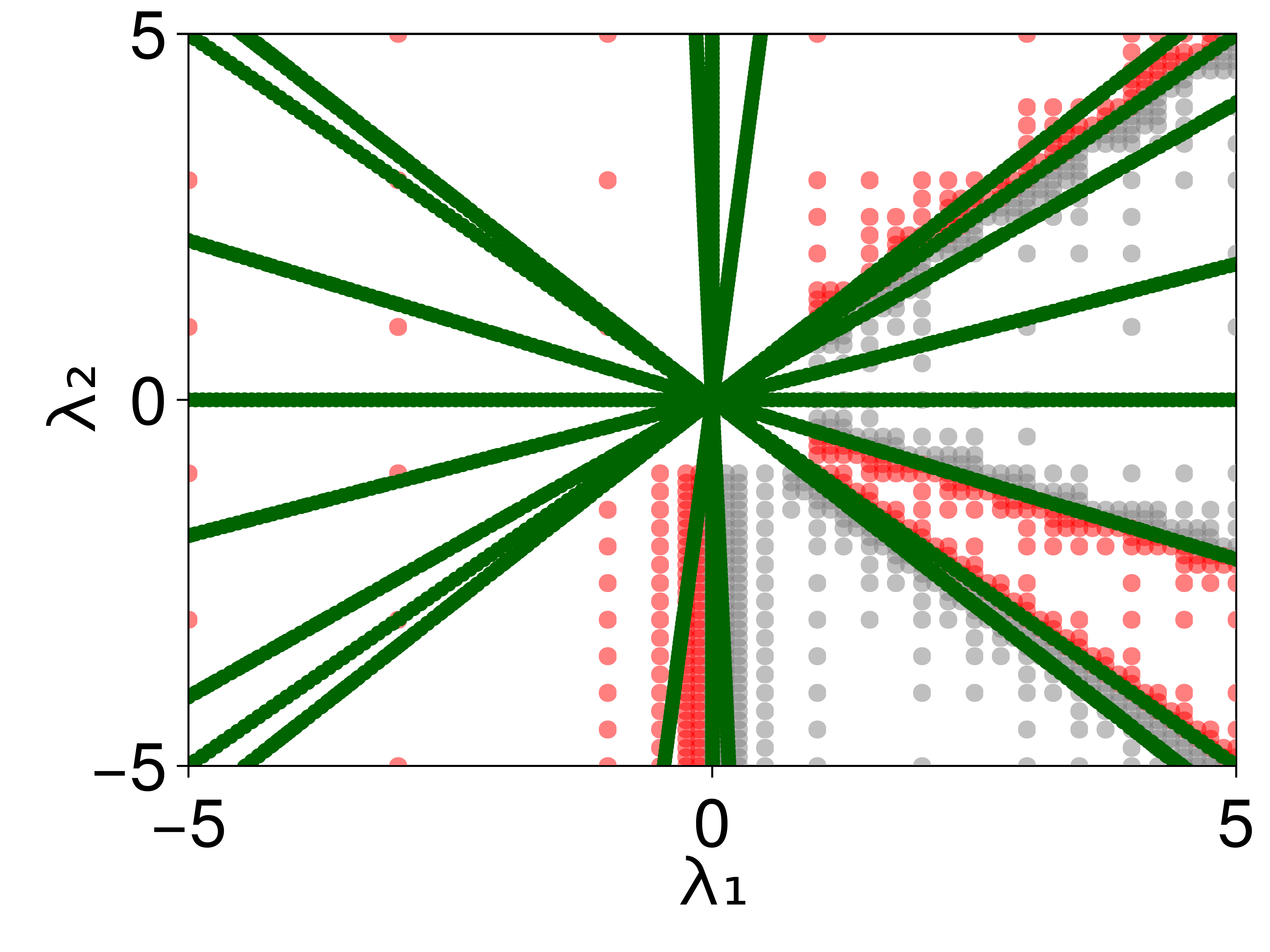}\panellabel{a}\end{overpic}\par
\panelcaption{\adjustbox{scale=0.835}{\footnotesize$\medmuskip=0mu \thickmuskip=.5mu \thinmuskip=0mu\relax\begin{aligned}\dot{x} &= -\lambda_2xz + \lambda_1y^2\\[-1pt] \dot{y} &= -\lambda_2xy + (-\lambda_1 + \lambda_2)xz + \lambda_1y^2\\[-1pt] \dot{z} &= (\lambda_1 + \lambda_2)x^2 -\lambda_1xz -\lambda_1z^2\end{aligned}$}}
\end{minipage}
\hfill
\begin{minipage}[t]{0.48\columnwidth}
\centering
\begin{overpic}[width=\linewidth]{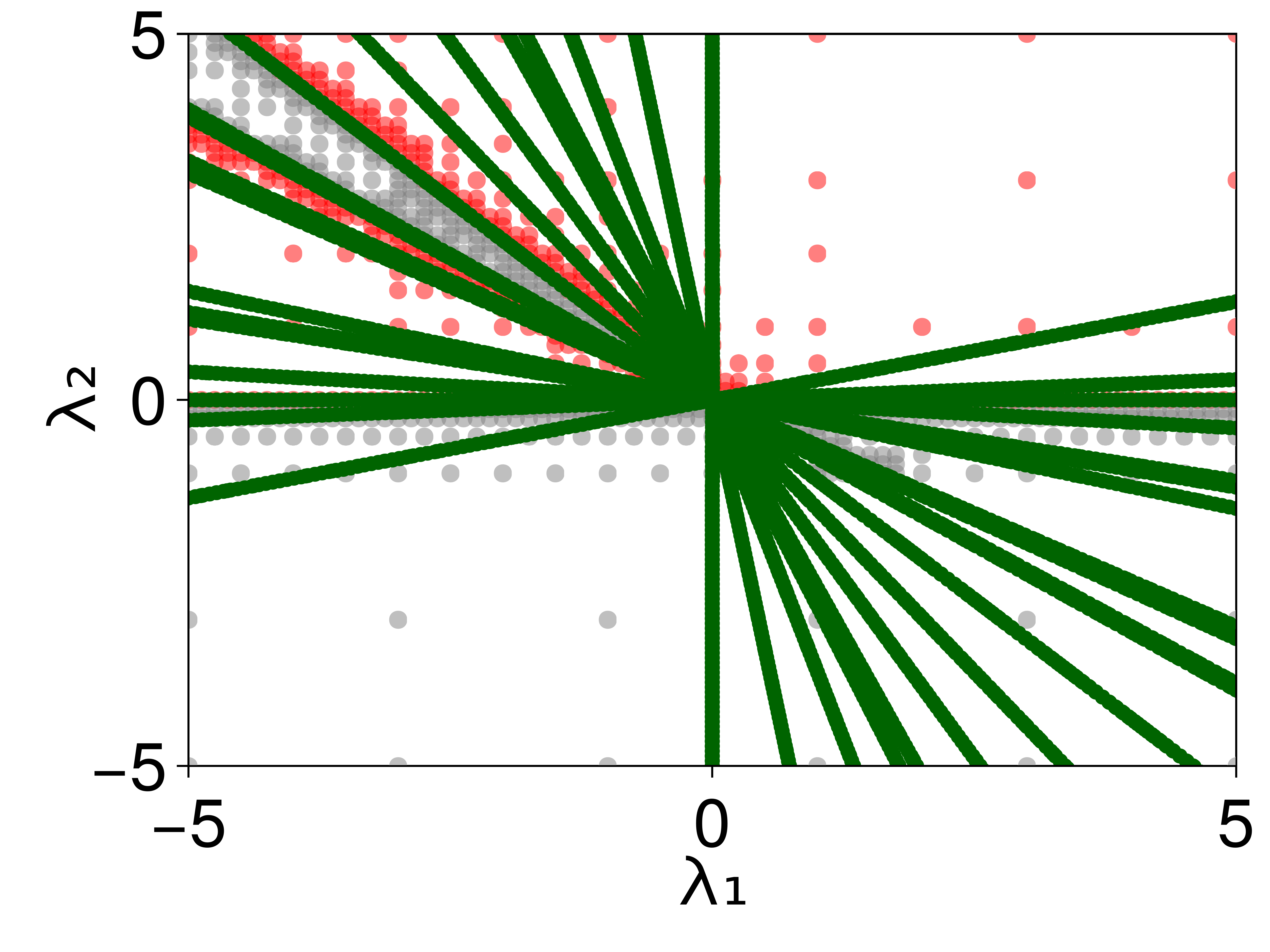}\panellabel{b}\end{overpic}\par
\panelcaption{\adjustbox{max width=\linewidth}{\normalsize$\medmuskip=0mu \thickmuskip=.5mu \thinmuskip=0mu\relax\begin{aligned}\dot{x} &= -\tfrac{11}{12}\lambda_1x^2 -\tfrac{7}{12}\lambda_1xy -\tfrac{11}{12}\lambda_2y^2 + \tfrac{1}{2}\lambda_2yz\\[1pt] \dot{y} &= \tfrac{7}{12}\lambda_2x^2 -\tfrac{5}{12}\lambda_2xy -\tfrac{1}{2}\lambda_1y^2 + \tfrac{1}{3}\lambda_1z^2\\[1pt] \dot{z} &= (-\tfrac{1}{2}\lambda_1 -\tfrac{3}{4}\lambda_2)x^2 + (\tfrac{11}{12}\lambda_1 + \tfrac{11}{12}\lambda_2)xy\\ &\quad -\tfrac{3}{4}\lambda_1y^2 + (\tfrac{5}{6}\lambda_1 + \tfrac{5}{12}\lambda_2)z^2\end{aligned}$}}
\end{minipage}
\par\vspace{2pt}
\begin{minipage}[t]{0.48\columnwidth}
\centering
\begin{overpic}[width=\linewidth]{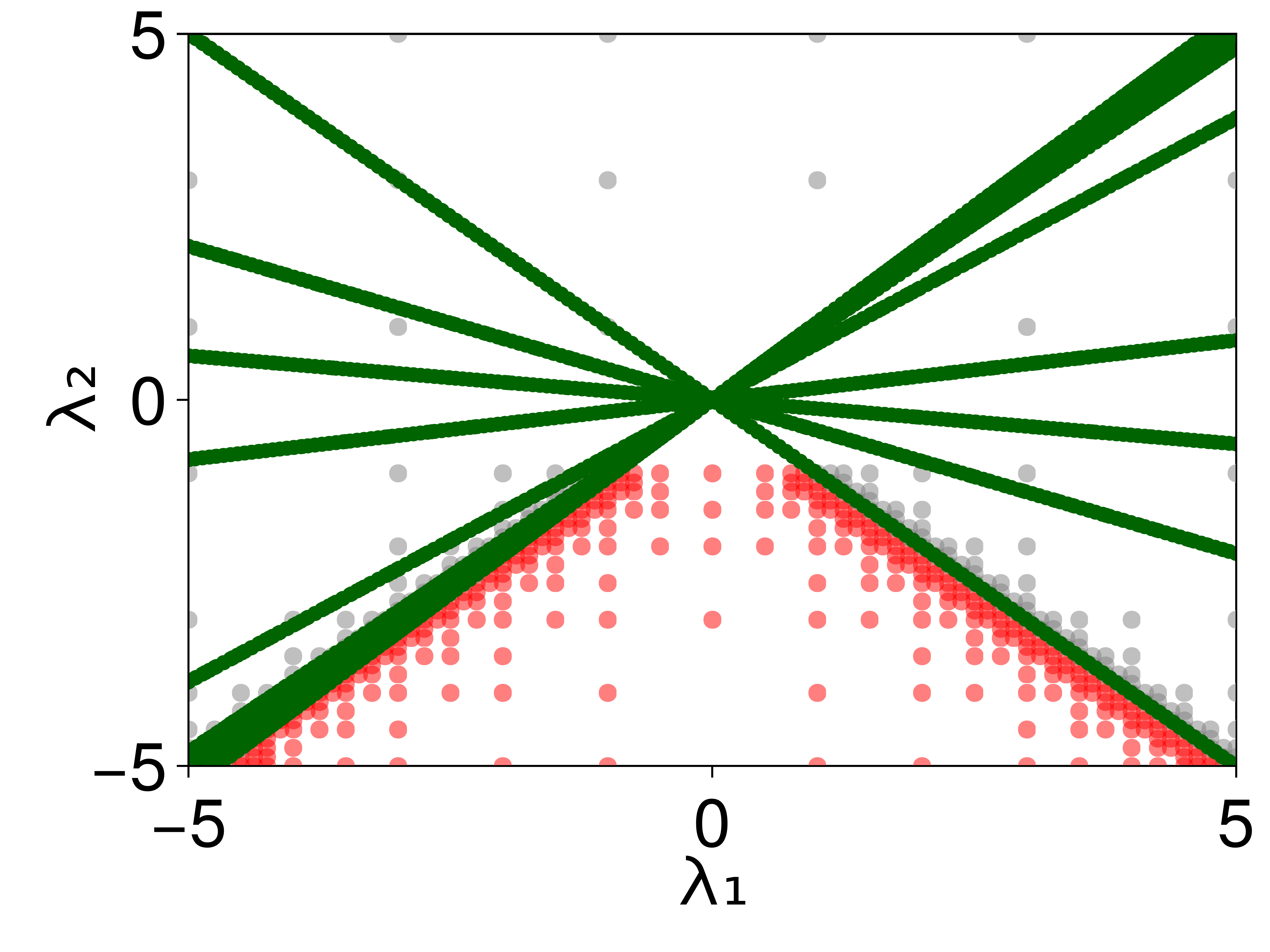}\panellabel{c}\end{overpic}\par
\panelcaption{\adjustbox{max width=\linewidth}{\normalsize$\medmuskip=0mu \thickmuskip=.5mu \thinmuskip=0mu\relax\begin{aligned}\dot{x} &= (\tfrac{1}{4}\lambda_1 + \tfrac{3}{4}\lambda_2)x^2 + \tfrac{7}{12}\lambda_1xy\\ &\quad -\tfrac{1}{4}\lambda_2xz + (-\tfrac{1}{6}\lambda_1 -\tfrac{2}{3}\lambda_2)y^2\\ &\quad + (\tfrac{1}{3}\lambda_1 -\tfrac{3}{4}\lambda_2)yz + (\tfrac{5}{6}\lambda_1 -\tfrac{1}{6}\lambda_2)z^2\\[1pt] \dot{y} &= (\tfrac{1}{4}\lambda_1 + \tfrac{1}{4}\lambda_2)x^2 + (-\tfrac{7}{12}\lambda_1 -\tfrac{11}{12}\lambda_2)xy\\ &\quad + (-\tfrac{1}{12}\lambda_1 + \tfrac{1}{6}\lambda_2)xz + (\tfrac{7}{12}\lambda_1 -\tfrac{1}{6}\lambda_2)y^2\\ &\quad + (-\tfrac{5}{12}\lambda_1 + \tfrac{1}{6}\lambda_2)yz -\tfrac{5}{6}\lambda_2z^2\\[1pt] \dot{z} &= (-\tfrac{1}{3}\lambda_1 -\tfrac{7}{12}\lambda_2)x^2 + (\tfrac{1}{12}\lambda_1 + \tfrac{1}{2}\lambda_2)xy\\ &\quad + (-\tfrac{1}{4}\lambda_1 + \tfrac{7}{12}\lambda_2)xz + (-\tfrac{7}{12}\lambda_1 -\tfrac{3}{4}\lambda_2)y^2\\ &\quad -\tfrac{1}{3}\lambda_2yz + (\tfrac{1}{3}\lambda_1 -\tfrac{5}{6}\lambda_2)z^2\end{aligned}$}}
\end{minipage}
\hfill
\begin{minipage}[t]{0.48\columnwidth}
\centering
\begin{overpic}[width=\linewidth]{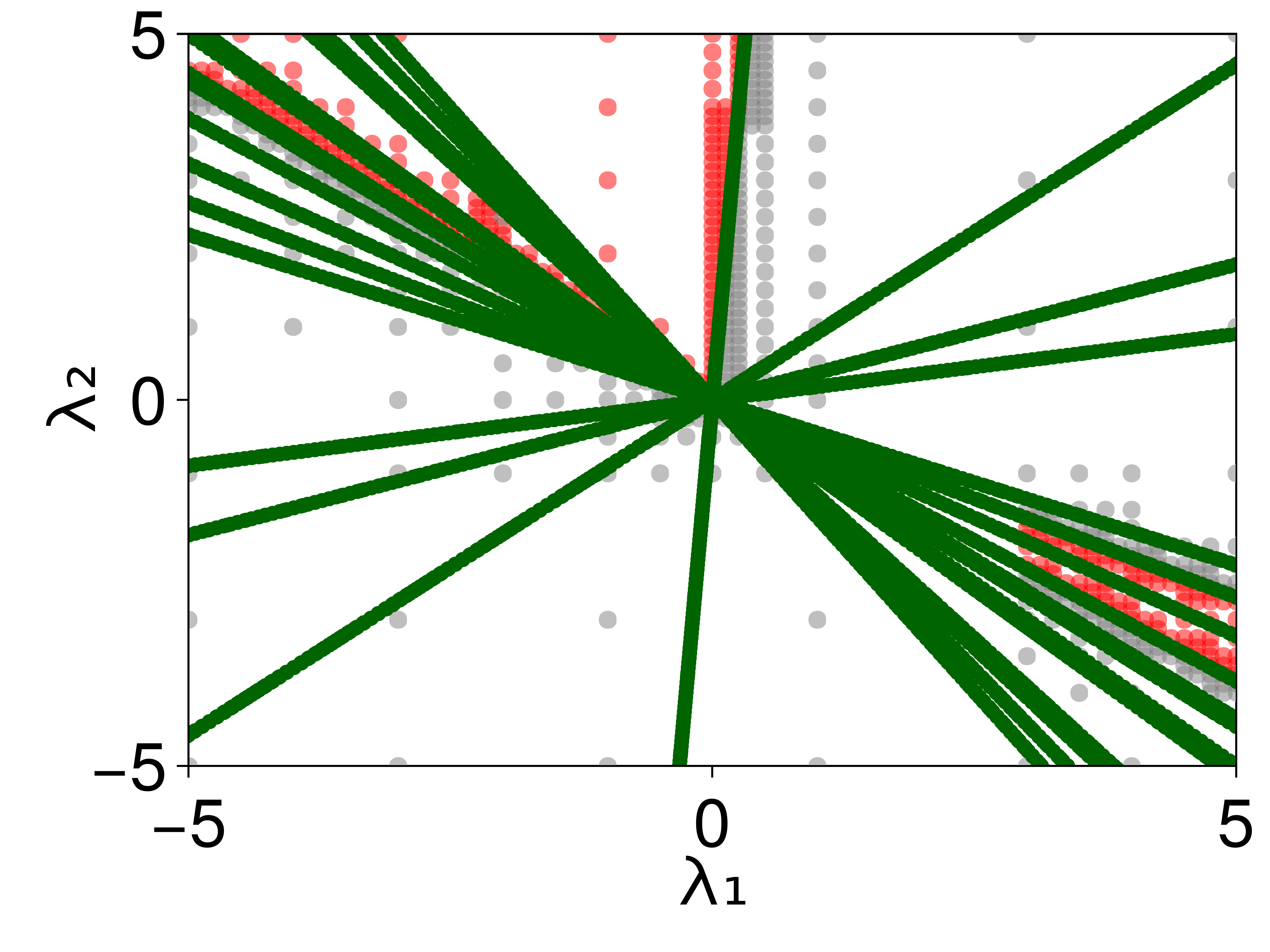}\panellabel{d}\end{overpic}\par
\panelcaption{\adjustbox{max width=\linewidth}{\normalsize$\medmuskip=0mu \thickmuskip=.5mu \thinmuskip=0mu\relax\begin{aligned}\dot{x} &= (-\tfrac{1}{4}\lambda_1 -\tfrac{5}{12}\lambda_2)x^2 + \tfrac{7}{12}\lambda_1xy\\ &\quad + (-\tfrac{1}{6}\lambda_1 -\tfrac{11}{12}\lambda_2)xz + (\tfrac{1}{6}\lambda_1 + \tfrac{1}{4}\lambda_2)y^2\\ &\quad + (-\tfrac{1}{6}\lambda_1 + \tfrac{1}{3}\lambda_2)yz + (-\tfrac{11}{12}\lambda_1 -\tfrac{1}{2}\lambda_2)z^2\\[1pt] \dot{y} &= (\tfrac{1}{2}\lambda_1 -\tfrac{1}{4}\lambda_2)x^2 + (-\tfrac{1}{2}\lambda_1 -\tfrac{1}{2}\lambda_2)xy\\ &\quad + (-\tfrac{1}{12}\lambda_1 + \tfrac{7}{12}\lambda_2)xz + (\tfrac{1}{6}\lambda_1 + \tfrac{11}{12}\lambda_2)y^2\\ &\quad + \tfrac{3}{4}\lambda_2yz + (-\tfrac{2}{3}\lambda_1 + \tfrac{11}{12}\lambda_2)z^2\\[1pt] \dot{z} &= (\tfrac{1}{3}\lambda_1 + \lambda_2)x^2 + (\tfrac{1}{3}\lambda_1 -\tfrac{1}{6}\lambda_2)xy\\ &\quad + \tfrac{1}{2}\lambda_1xz + (\tfrac{11}{12}\lambda_1 + \tfrac{11}{12}\lambda_2)y^2\\ &\quad + (-\tfrac{5}{6}\lambda_1 + \tfrac{5}{6}\lambda_2)yz + (\tfrac{3}{4}\lambda_1 + \tfrac{5}{6}\lambda_2)z^2\end{aligned}$}}
\end{minipage}
\par
\par\vspace{-1pt}
\centerline{\includegraphics[width=\columnwidth,clip,trim=0 12 0 12]{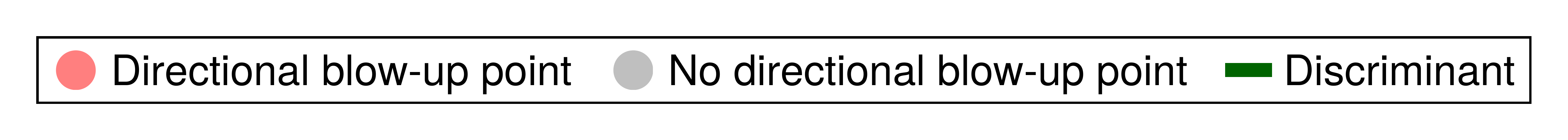}}
\vspace{-15pt}
\caption{Generic 3-variable quadratic: numerically sampled positive-orthant blow-up over $[-5,5]^2$.}
\label{fig:gen3d}
\end{figure}

\subsection{Parametric Keyfitz--Kranzer}
We study a parametrized Keyfitz--Kranzer system~\cite{keyfitz2012conserving,keyfitz1995spaces,kranzer1990strictly,keyfitz2011singular}
$$\dot u = a\,u^2 + b\,v, \qquad \dot v = f\,u^3 + g\,u v - u.$$
A similar system can be derived from a gas model~\cite{keyfitz2011singular}, where $a,b,f,g$ all depend on a single constant $\gamma$.
The highest degree components of each equation have different degrees, which degenerates our framework. We instead use quasi-homogeneous compactification~\cite{matsue2020numerical}, which replaces the normalization factor $\kappa$ defined in \cref{subsec:eqinf} by 
$\hat{\kappa} = \left(\sum_{i=1}^{d} x_i^{2 \beta_i}\right)^{1/2c}$ in ref.~\onlinecite[def.~2.3]{matsue2020numerical}, where $\beta_1,\ldots,\beta_d,c$ are constants that depend on the leading degrees in the system. This leads to an equivalent compactification framework to the one laid out in \cref{sec:background}, including new quasi-homogeneous highest degree terms that govern the behavior on the boundary of the compactified space~\cite{matsue2020numerical}. From there we can proceed with our discriminant-based decomposition. 
We make four two-parameter slices for visualization in \cref{fig:KK}.

\begin{figure}
\centering
\renewcommand{\panelcaption}[1]{\par\vspace{-6pt}{\small\centering #1\par}}
\begin{minipage}[t]{0.48\columnwidth}
\centering
\begin{overpic}[width=\linewidth]{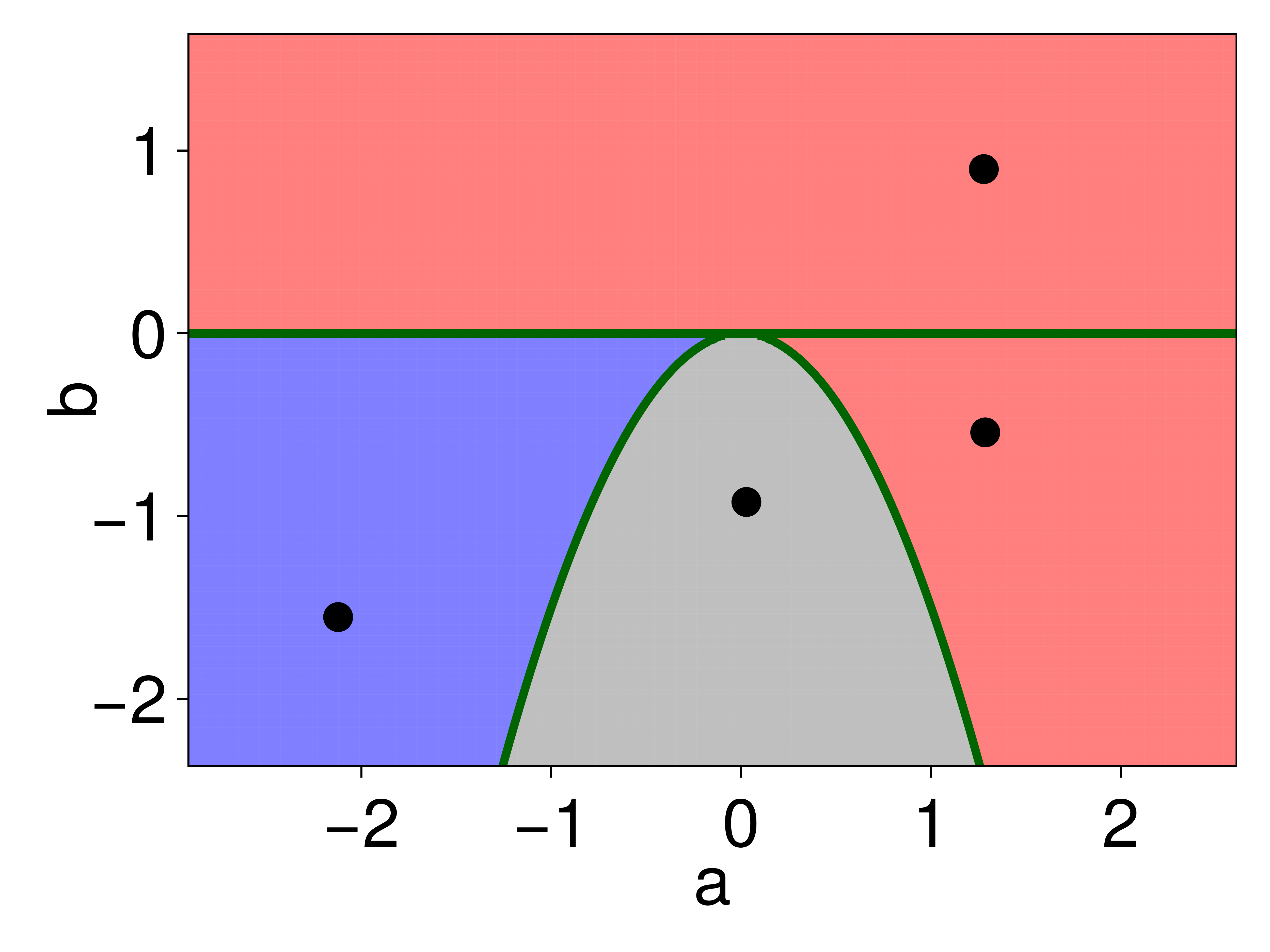}\panellabel{a}\end{overpic}\par
\panelcaption{$\dot u = a\,u^2 + b\,v,\quad \dot v = \tfrac13 u^3$}
\end{minipage}
\hfill
\begin{minipage}[t]{0.48\columnwidth}
\centering
\begin{overpic}[width=\linewidth]{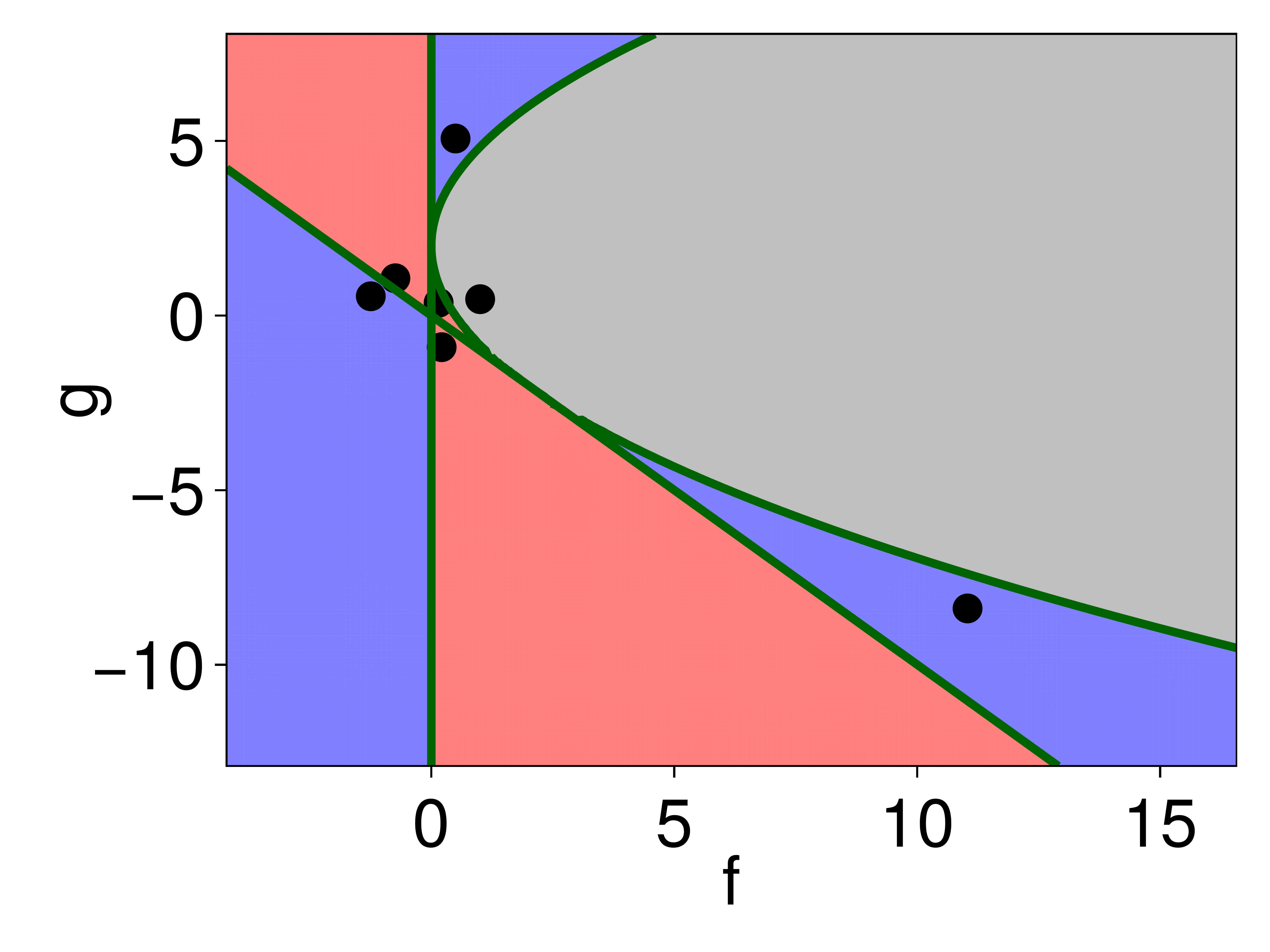}\panellabel{b}\end{overpic}\par
\panelcaption{$\dot u = u^2 - v,\quad \dot v = f\,u^3 + g\,uv$}
\end{minipage}
\par\vspace{2pt}
\begin{minipage}[t]{0.48\columnwidth}
\centering
\begin{overpic}[width=\linewidth]{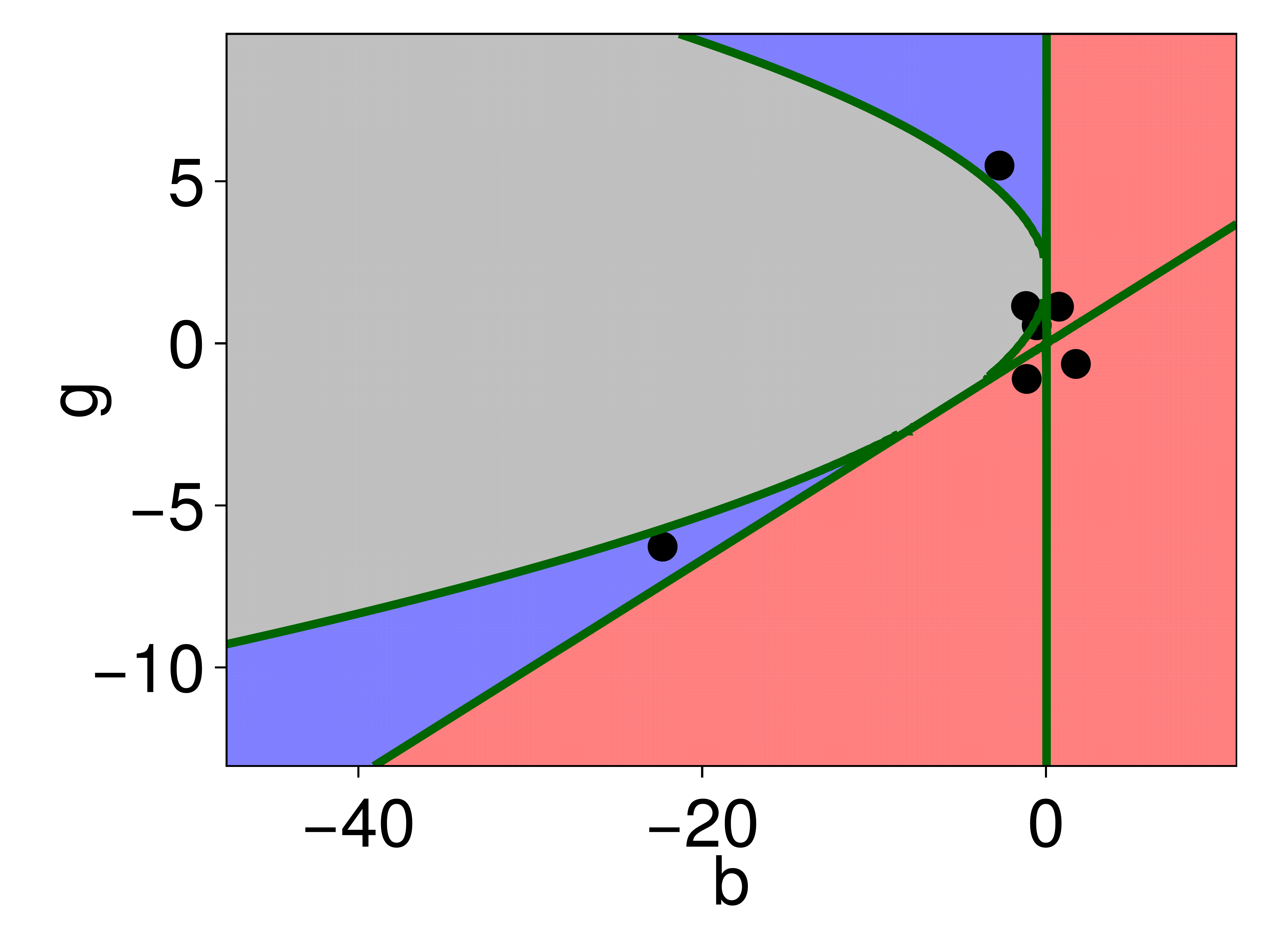}\panellabel{c}\end{overpic}\par
\panelcaption{$\dot u = u^2 + b\,v,\quad \dot v = \tfrac13 u^3 + g\,uv$}
\end{minipage}
\hfill
\begin{minipage}[t]{0.48\columnwidth}
\centering
\begin{overpic}[width=\linewidth]{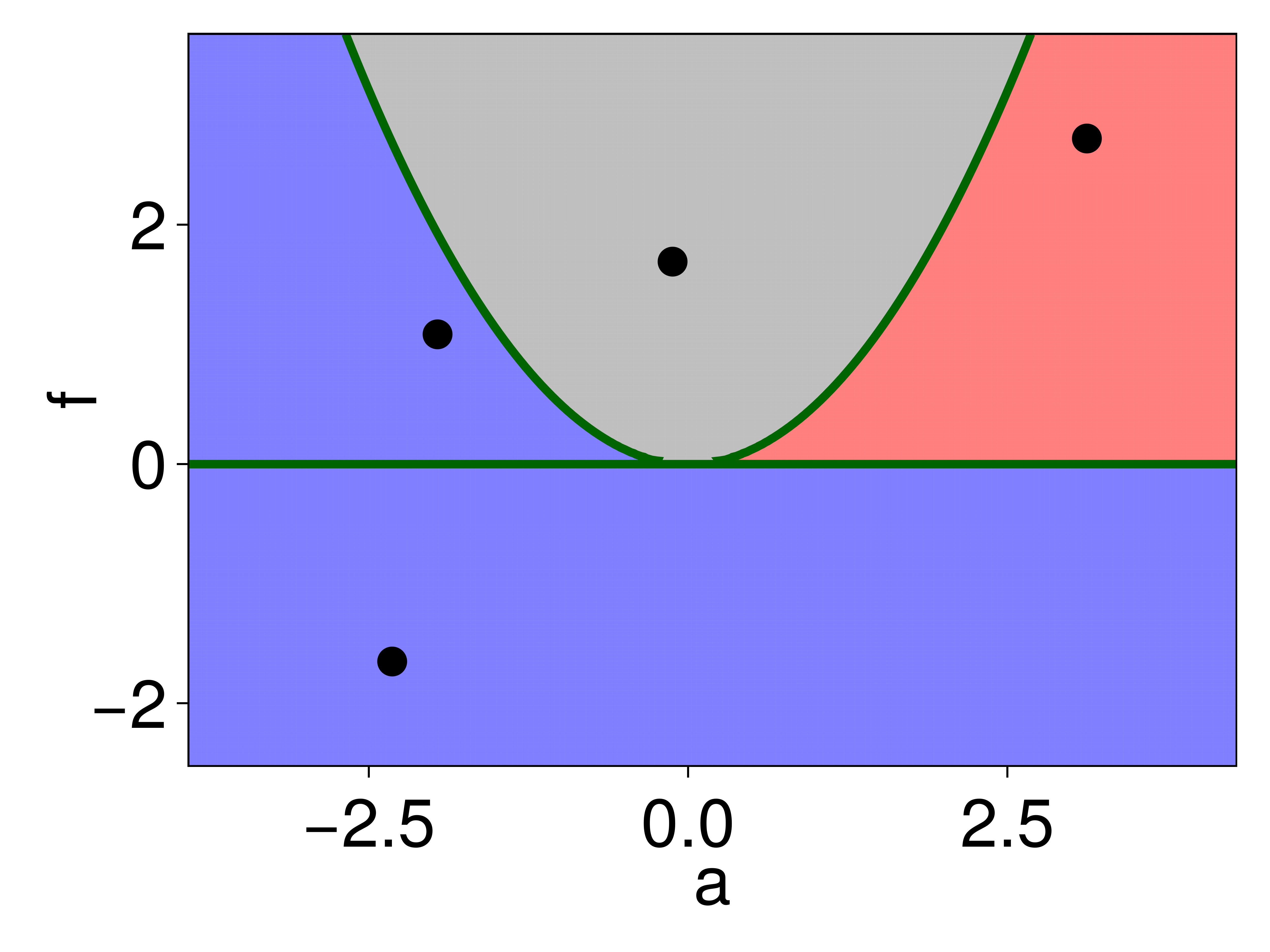}\panellabel{d}\end{overpic}\par
\panelcaption{$\dot u = a\,u^2 - v,\quad \dot v = f\,u^3$}
\end{minipage}
\par\vspace{-1pt}
\centerline{\includegraphics[width=\columnwidth,clip,trim=0 44 0 44]{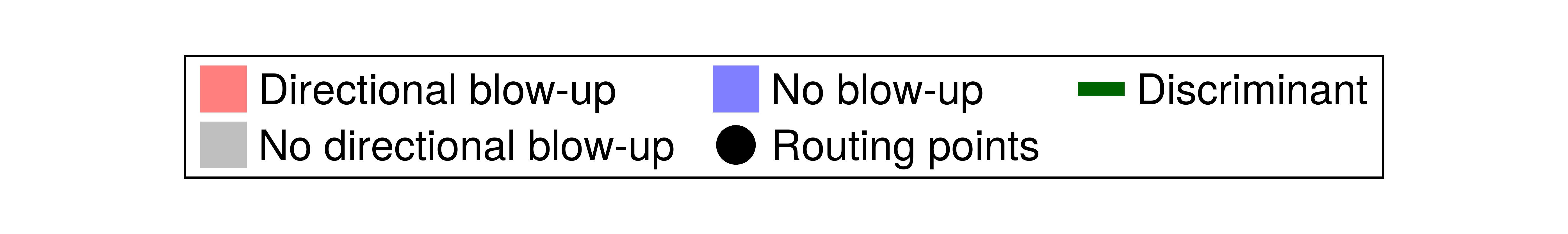}}
\vspace{-15pt}
\caption{Parametric Keyfitz--Kranzer: positive-orthant blow-up landscapes, with discriminant boundary only.}
\label{fig:KK}
\end{figure}

\section{Spiral Blow-up}
Our previous examples focus on directional blow-up, where $\frac{\xx(t)}{\|\xx(t)\|} \to \xx^* \in \R^d$, and  the compactified trajectory $(\zz(t),y(t))$ tends to an equilibrium on the sphere at infinity. But other types of blow-up are possible, so some regions can only be classified as ``no directional blow-up'', leaving other types of blow-up undetermined. In this section we explore spiral blow-up of planar systems.

On the circle at infinity $\zz(\theta)=(\cos\theta,\sin\theta)$, with unit tangent $T=(-\sin\theta,\cos\theta)$, define the tangential speed
\[
S(\theta)=\langle T, f_n(\zz)\rangle \;=\; \pm\|\tilde g\|, \quad \text{with } \tilde{g}(\zz) = g(\zz,0).
\]
A region is shaded gray if and only if $S(\theta)\neq 0$ for all $\theta$; then the boundary of the unit disc is a periodic orbit with $\dot\theta=c\,S(\theta)$ of constant sign. We study the stability of this periodic orbit by the Floquet exponent
\[
\Lambda = \oint \mu_\perp\,dT
= \int_0^{2\pi}\frac{\mu_\perp(\theta)}{|\dot\theta(\theta)|}\,d\theta,
 |\dot\theta|\propto \big\|f_n(\zz)-\langle \zz,f_n\rangle \zz\big\|.
\]
This gives the total distance that the linearization of the system moves towards the boundary during one orbit. Ref.~\onlinecite[Sec.~1.6,7.3]{dumortier2006qualitative} implies that
\(
\Lambda<0 \iff \text{boundary orbit stable} \iff \text{spiral blow-up}.
\)

To test, we study a generic two-dimensional cubic system.
Many regions cannot be classified as either directional blow-up or no blow-up. For each undetermined region in the parameter landscapes we sweep the parameter window with marching squares refinement. We classify every gray point by using the sign of $\Lambda$, and confirm by integrating the full ODE from several initial conditions. \Cref{fig:spiral-grid} overlays the result on the standard landscapes: the gray regions split cleanly into pink (spiral blow-up, $\Lambda<0$) and blue (no blow-up, $\Lambda\ge 0$) sub-regions.

\begin{figure}
\centering
\begin{minipage}[t]{0.48\columnwidth}
\centering
\begin{overpic}[width=\linewidth]{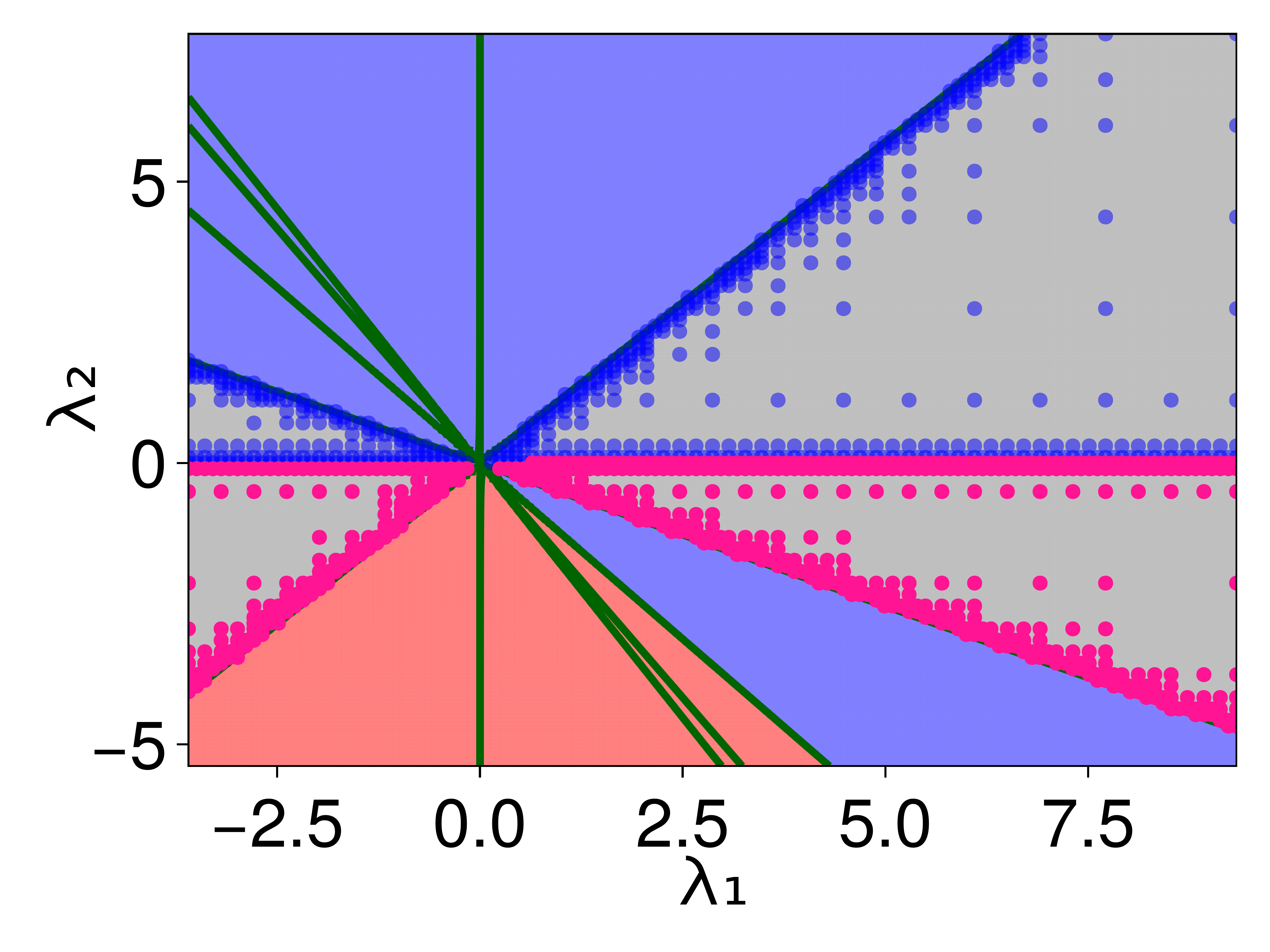}\panellabel{a}\end{overpic}\par
\panelcaption{\adjustbox{scale=0.61}{\normalsize$\medmuskip=0mu \thickmuskip=.5mu \thinmuskip=0mu\relax\begin{aligned}\dot{u} &= (-\tfrac{1}{12}\lambda_1 -\tfrac{2}{3}\lambda_2)uv^2 + (-\tfrac{5}{6}\lambda_1 -\tfrac{1}{2}\lambda_2)v^3\\[1pt] \dot{v} &= \tfrac{5}{12}\lambda_1u^3 -\tfrac{2}{3}\lambda_2u^2v -\tfrac{2}{3}\lambda_1uv^2 -\tfrac{1}{3}\lambda_2v^3\end{aligned}$}}
\end{minipage}
\hfill
\begin{minipage}[t]{0.48\columnwidth}
\centering
\begin{overpic}[width=\linewidth]{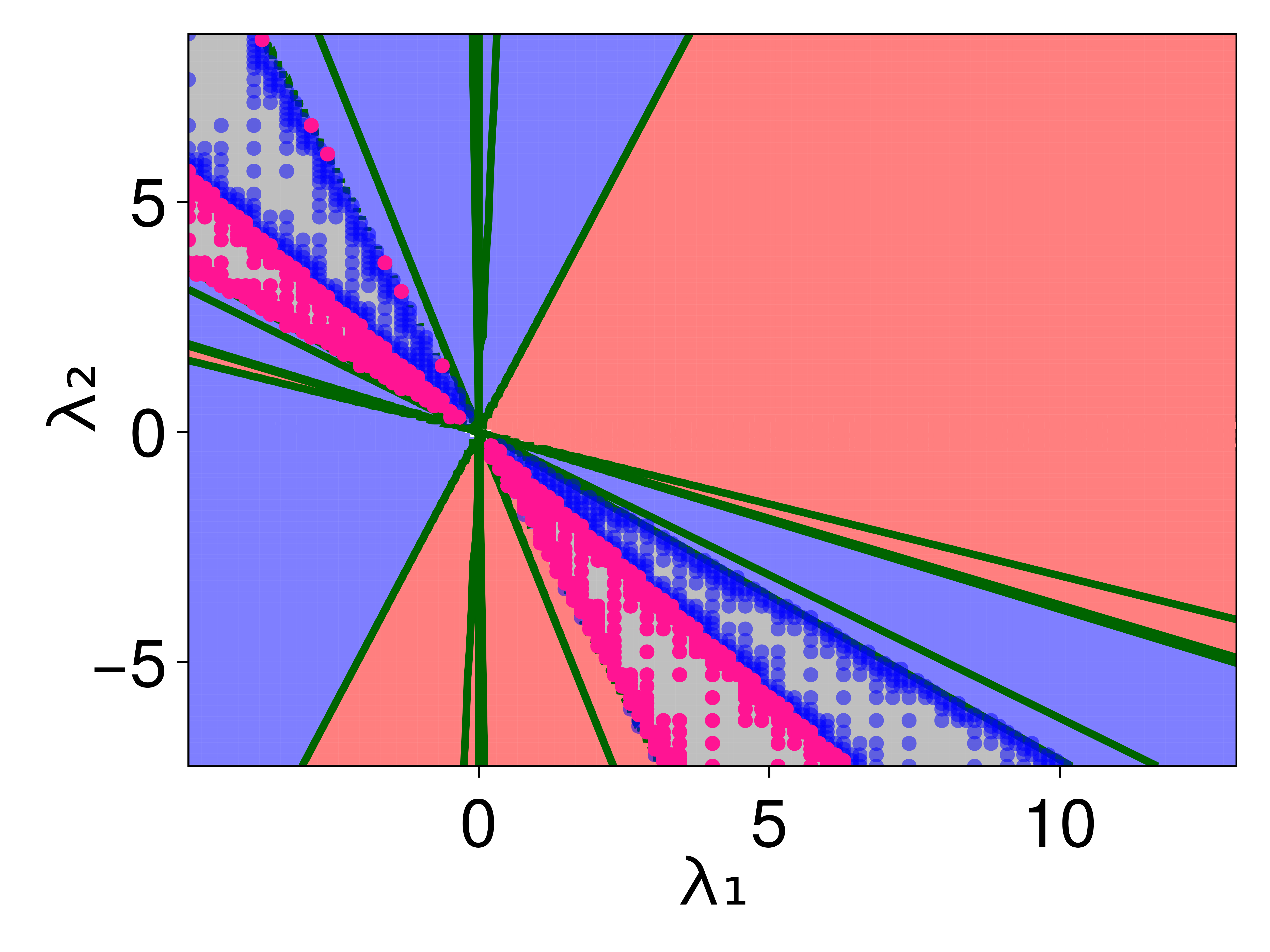}\panellabel{b}\end{overpic}\par
\panelcaption{\adjustbox{scale=0.61}{\normalsize$\medmuskip=0mu \thickmuskip=.5mu \thinmuskip=0mu\relax\begin{aligned}\dot{u} &= -\tfrac{1}{6}\lambda_2u^3 -\tfrac{7}{12}\lambda_1uv^2 + \tfrac{7}{12}\lambda_1v^3\\[1pt] \dot{v} &= (\tfrac{1}{4}\lambda_1 + \tfrac{2}{3}\lambda_2)u^3 + \tfrac{7}{12}\lambda_2u^2v -\tfrac{1}{2}\lambda_2uv^2 -\tfrac{1}{4}\lambda_2v^3\end{aligned}$}}
\end{minipage}
\par\vspace{2pt}
\begin{minipage}[t]{0.48\columnwidth}
\centering
\begin{overpic}[width=\linewidth]{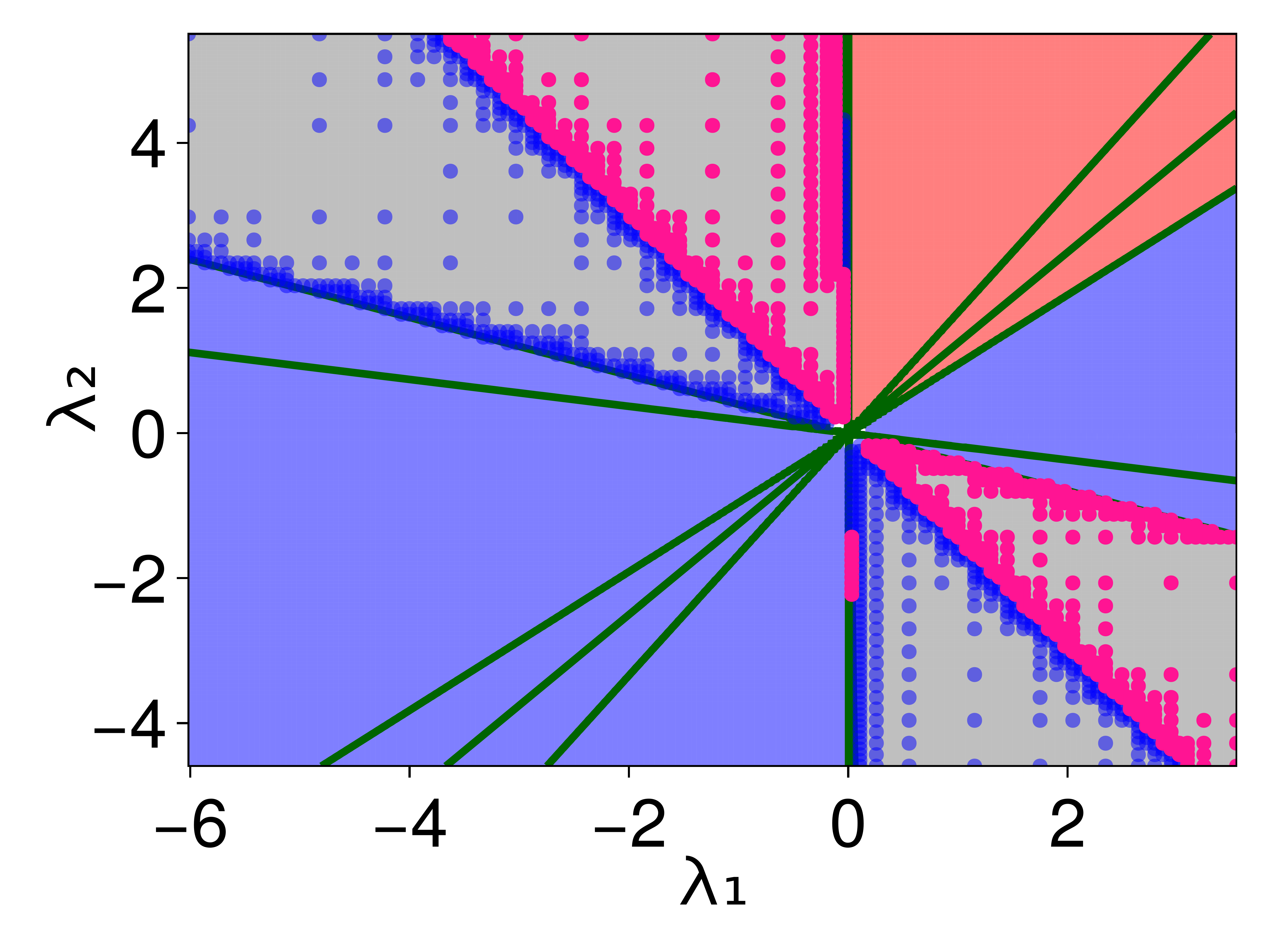}\panellabel{c}\end{overpic}\par
\panelcaption{\adjustbox{scale=0.61}{\normalsize$\medmuskip=0mu \thickmuskip=.5mu \thinmuskip=0mu\relax\begin{aligned}\dot{u} &= (\tfrac{1}{2}\lambda_1 + \tfrac{2}{3}\lambda_2)u^3 + (-\tfrac{1}{3}\lambda_1 -\lambda_2)u^2v\\ &\quad + (\tfrac{11}{12}\lambda_1 -\tfrac{1}{4}\lambda_2)uv^2 + \tfrac{2}{3}\lambda_1v^3\\[1pt] \dot{v} &= (-\tfrac{5}{12}\lambda_1 + \tfrac{1}{3}\lambda_2)u^3 + (\tfrac{1}{12}\lambda_1 + \tfrac{1}{12}\lambda_2)u^2v\\ &\quad + (\tfrac{1}{3}\lambda_1 + \tfrac{3}{4}\lambda_2)uv^2\end{aligned}$}}
\end{minipage}
\hfill
\begin{minipage}[t]{0.48\columnwidth}
\centering
\begin{overpic}[width=\linewidth]{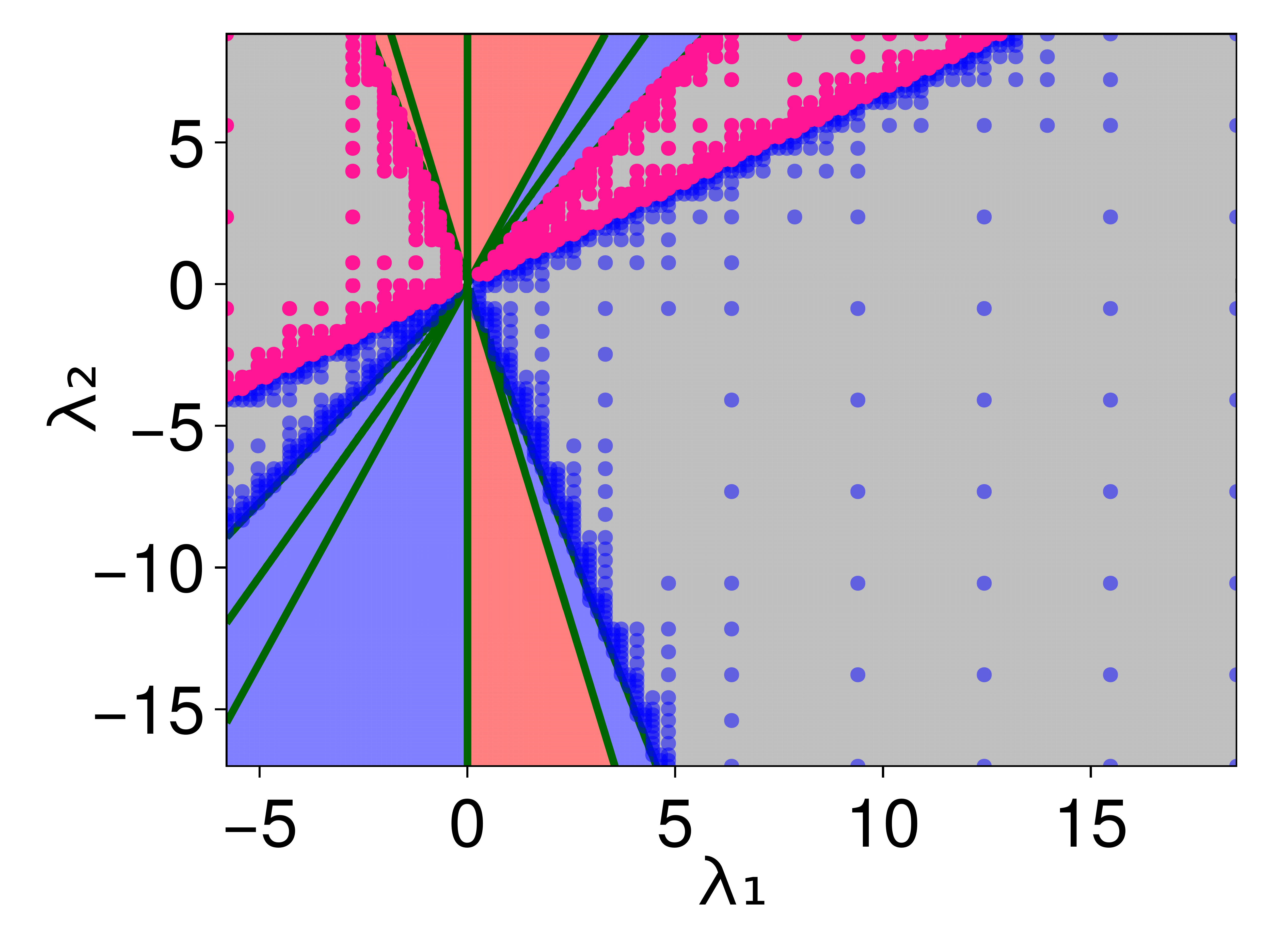}\panellabel{d}\end{overpic}\par
\panelcaption{\adjustbox{scale=0.61}{\normalsize$\medmuskip=0mu \thickmuskip=.5mu \thinmuskip=0mu\relax\begin{aligned}\dot{u} &= (-\tfrac{1}{6}\lambda_1 + \tfrac{3}{4}\lambda_2)u^3 + (\tfrac{1}{2}\lambda_1 -\tfrac{1}{6}\lambda_2)u^2v\\ &\quad + (-\tfrac{7}{12}\lambda_1 -\tfrac{1}{4}\lambda_2)uv^2 + \tfrac{5}{6}\lambda_1v^3\\[1pt] \dot{v} &= (-\tfrac{2}{3}\lambda_1 + \tfrac{1}{4}\lambda_2)u^3 -\tfrac{1}{12}\lambda_2u^2v\\ &\quad -\tfrac{5}{12}\lambda_1uv^2 -\tfrac{1}{4}\lambda_2v^3\end{aligned}$}}
\end{minipage}
\hfill
\par\vspace{-1pt}
\centerline{\includegraphics[width=\columnwidth,clip,trim=0 44 0 44]{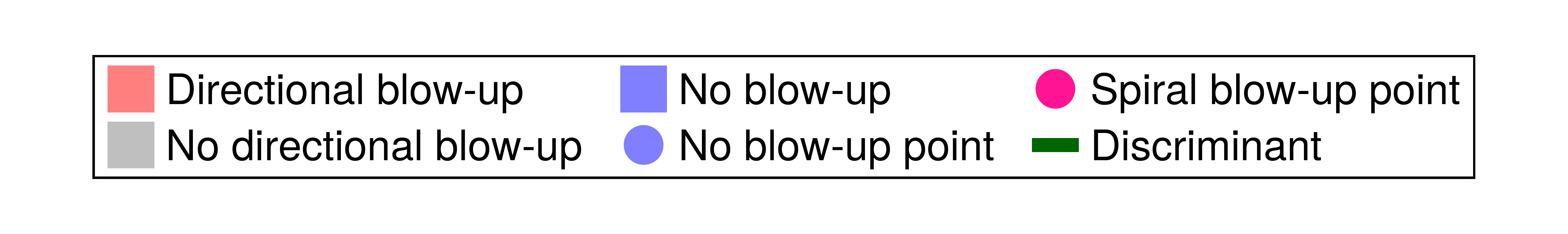}}
\vspace{-15pt}
\caption{Generic two-dimensional cubic landscapes, gray regions resolved by the boundary-integral test and simulation.}
\label{fig:spiral-grid}
\end{figure}

\section{Moehlis--Knobloch Bursting}

Moehlis and Knobloch~\cite{moehlis2000bursts,moehlis1998forced,moehlis2001effect} use equilibria at infinity to study the phenomenon of bursting: a recurrent, finite-time excursion to large amplitude. An equilibrium at infinity ($y=0$) that is radially attracting yet tangentially unstable gives a trajectory that grows to a large but finite magnitude along the stable manifold, is eventually ejected along the unstable manifold, and returns to smaller magnitude. We use the stability discriminant $X_S$ to distinguish between stable and unstable equilibria at infinity, i.e. between stable and unstable equilibria of \cref{eq:tausys,eq:ytausys} with $y=0$. We can then sort between sustained blow-up (red), when there is a radially attracting and tangentially stable equilibrium on the sphere at infinity, versus bursting (orange), when the equilibrium is radially attracting but tangentially a saddle.

We use the model in ref.~\onlinecite[Eq.~1,2]{moehlis2000bursts}. After a coordinate transformation, as in ref.~\onlinecite[Eq.~7--10]{moehlis2000bursts}, we polynomialize the resulting system by adjoining $r=\sqrt{u^2+v^2+w^2}$. Our version of the ``perfect system'' from ref.~\onlinecite{moehlis2000bursts} is then
\begin{equation} \label{eq:knobloch}
\begin{aligned}
\dot{u} &= pru - (B_I+C_I)vw \\
\dot{v} &= qrv + (B_I-C_I)uw \\
\dot{w} &= srw + 2C_Iuv \\
\dot{r} &= pu^2 + qv^2 + sw^2,
\end{aligned}
\end{equation}
with $p=2A_R+B_R+C_R, q=2A_R+B_R-C_R, s=2(A_R+B_R)$.
The physical dynamics lives on the set $r^2=u^2+v^2+w^2$, so we include this equation as an extra input to the discriminant-elimination step. We sweep the two imaginary coefficients $(B_I,C_I)$ at fixed real coefficients $(A_R,B_R,C_R)$, given in each panel of \cref{fig:knobloch}.

\begin{figure}
\centering
\renewcommand{\panelcaption}[1]{\par\vspace{-6pt}{\small\centering #1\par}}
\begin{minipage}[t]{0.48\columnwidth}\centering
\begin{overpic}[width=\linewidth]{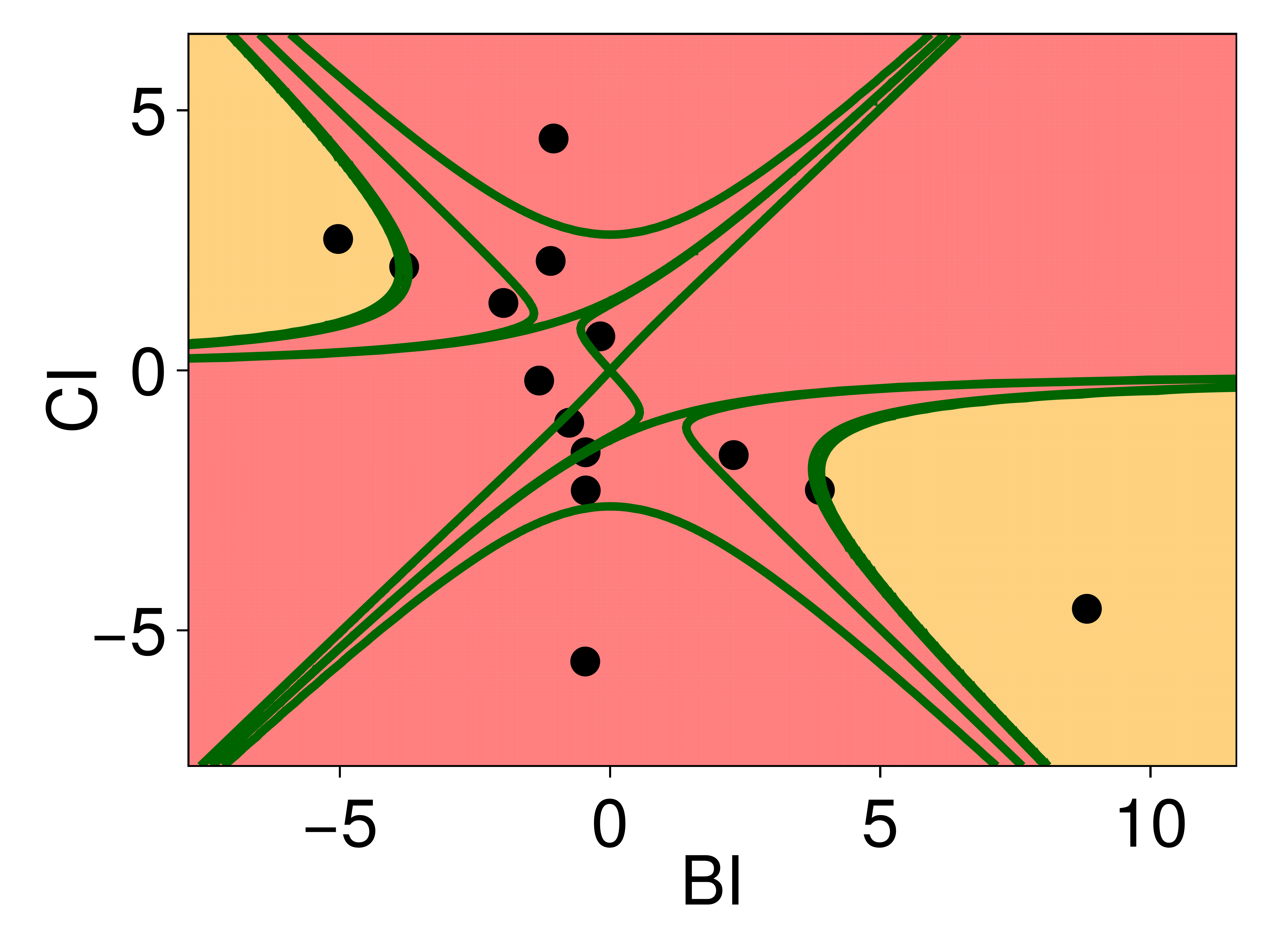}\panellabel{a}\end{overpic}\par
\panelcaption{\adjustbox{max width=\linewidth}{\footnotesize$(A_R,B_R,C_R)=(1,\,-2.8,\,-1)$}\\{\scriptsize ref.~\onlinecite[Fig.~1]{moehlis2000bursts}}}
\end{minipage}\hfill
\hfill
\begin{minipage}[t]{0.48\columnwidth}\centering
\begin{overpic}[width=\linewidth]{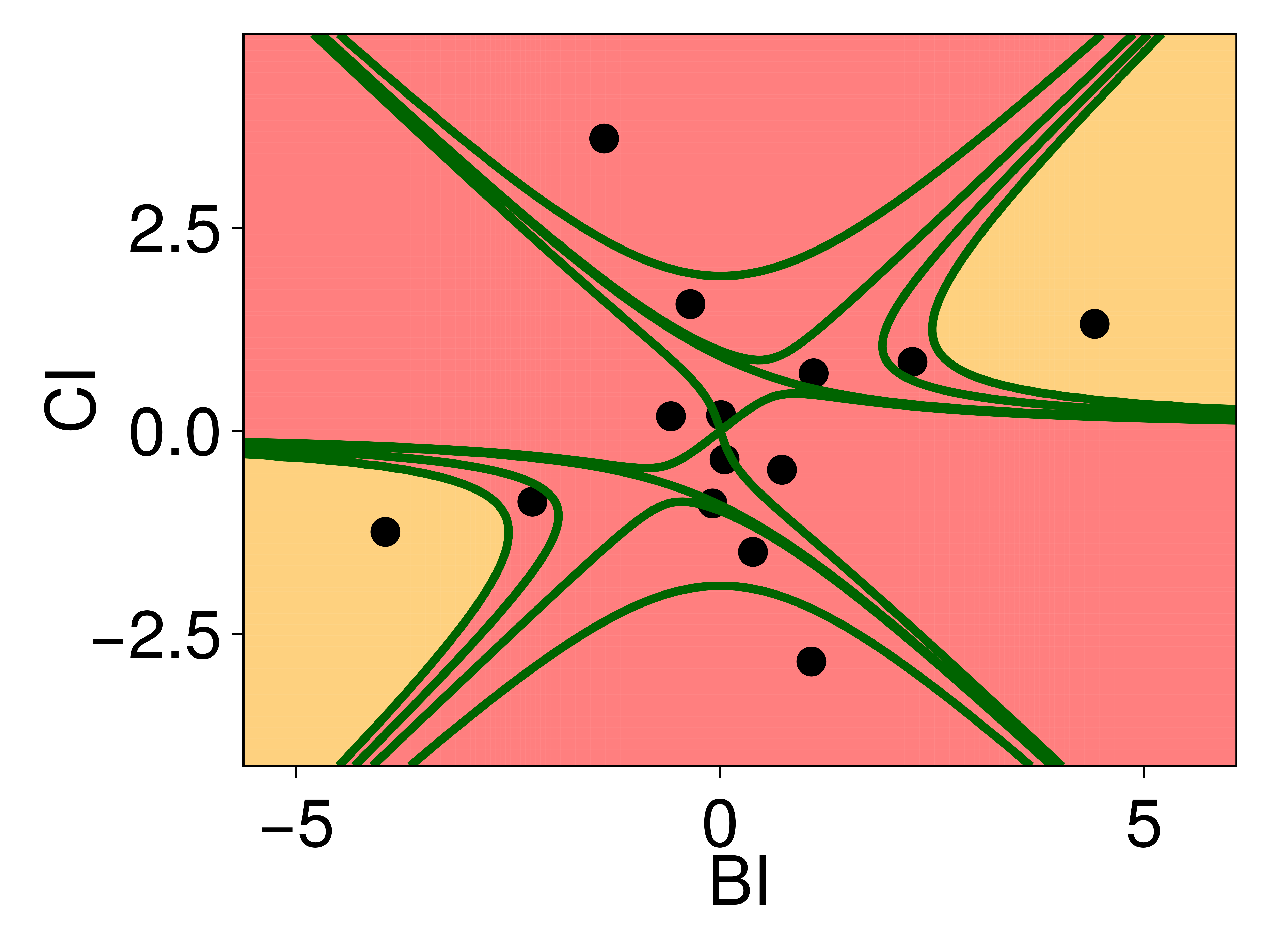}\panellabel{b}\end{overpic}\par
\panelcaption{\adjustbox{max width=\linewidth}{\footnotesize$(A_R,B_R,C_R)=(1,\,-2,\,0.6)$}}
\end{minipage}
\par\vspace{2pt}
\begin{minipage}[t]{0.48\columnwidth}\centering
\begin{overpic}[width=\linewidth]{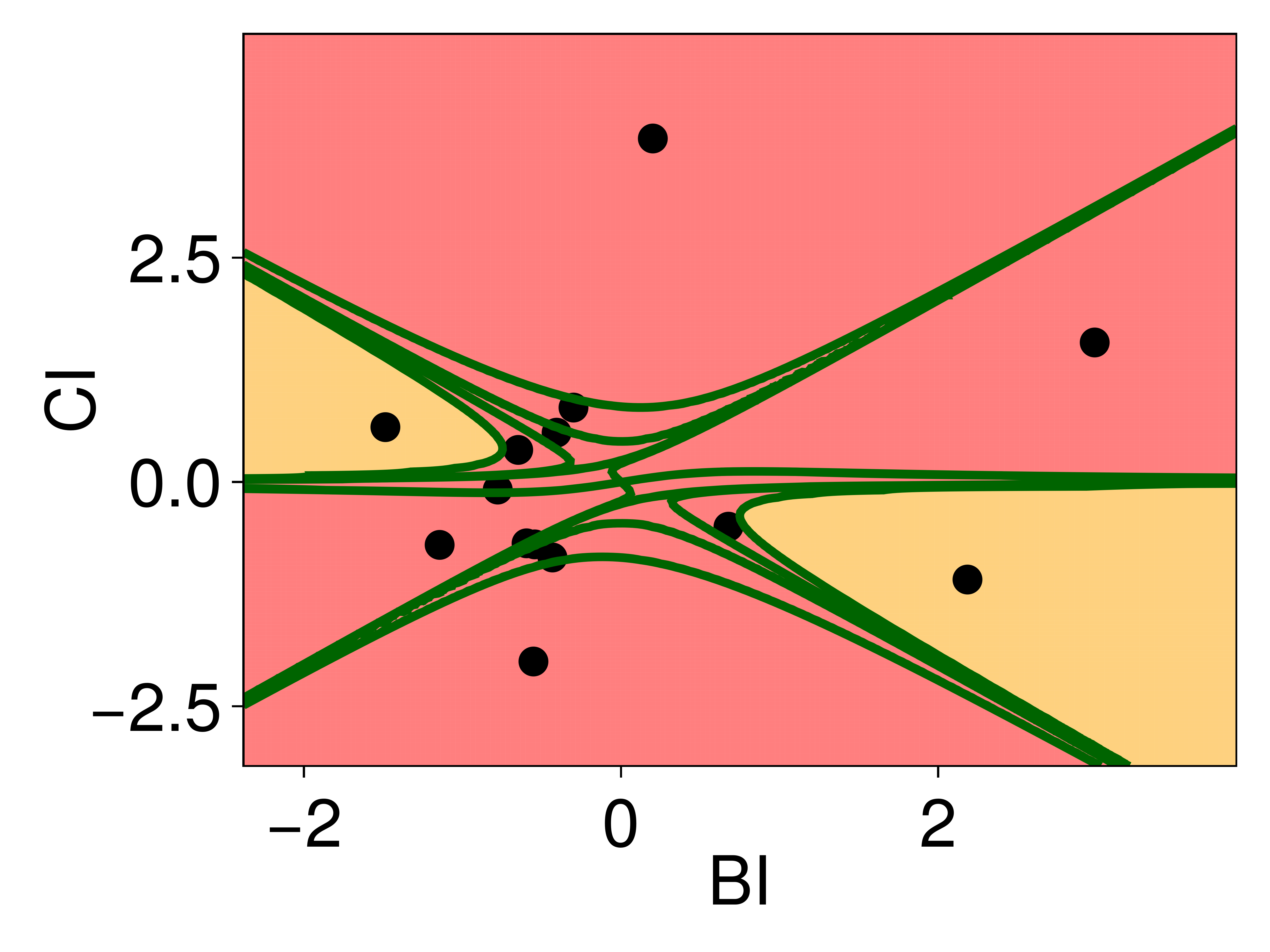}\panellabel{c}\end{overpic}\par
\panelcaption{\adjustbox{max width=\linewidth}{\footnotesize$(A_R,B_R,C_R)=(1,\,-0.5,\,-0.2)$}}
\end{minipage}\hfill
\hfill
\begin{minipage}[t]{0.48\columnwidth}\centering
\begin{overpic}[width=\linewidth]{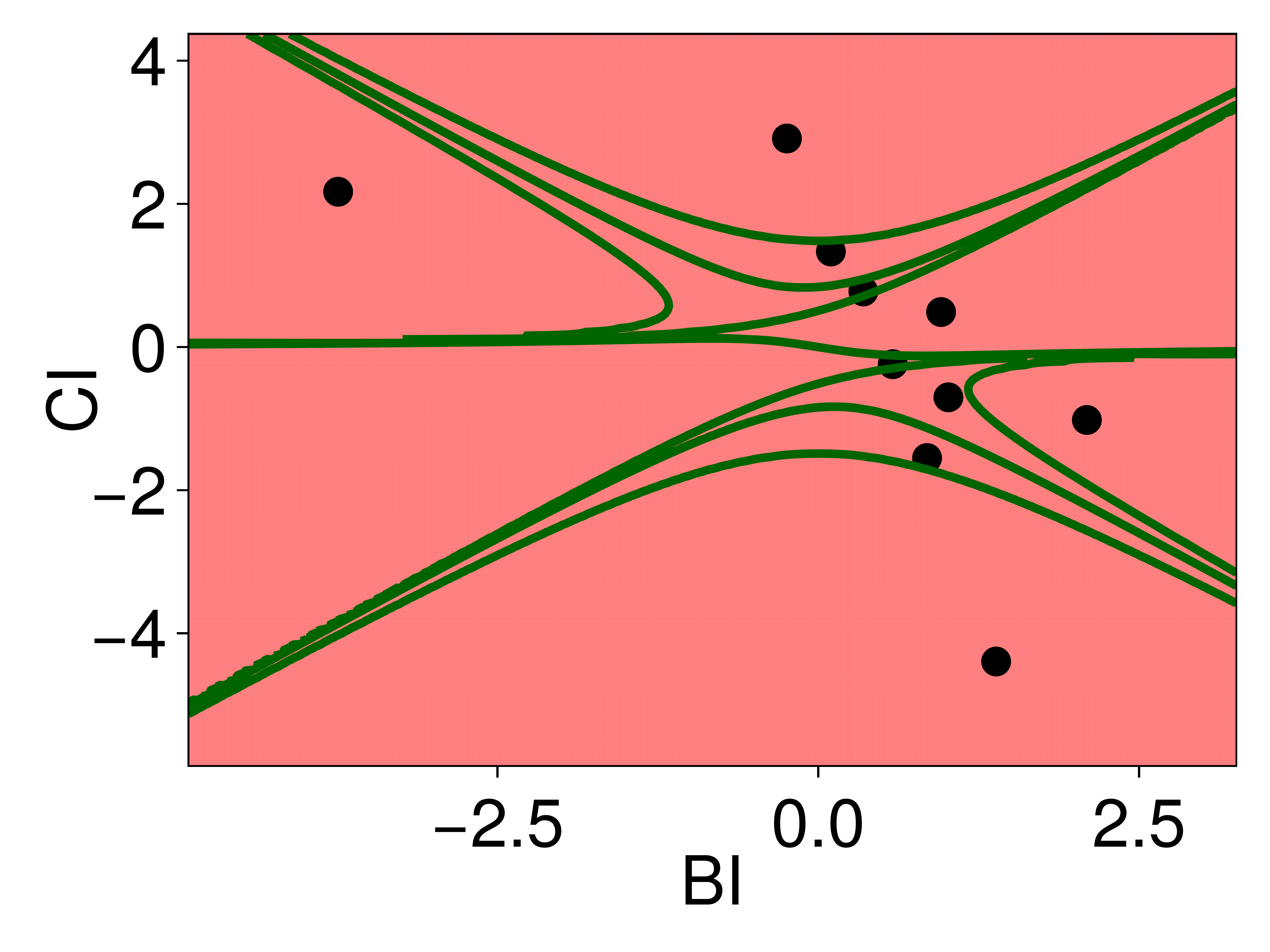}\panellabel{d}\end{overpic}\par
\panelcaption{\adjustbox{max width=\linewidth}{\footnotesize$(A_R,B_R,C_R)=(1,\,-1.5,\,-0.2)$}}
\end{minipage}
\par\vspace{-1pt}
\centerline{\includegraphics[width=\columnwidth,clip,trim=0 44 0 44]{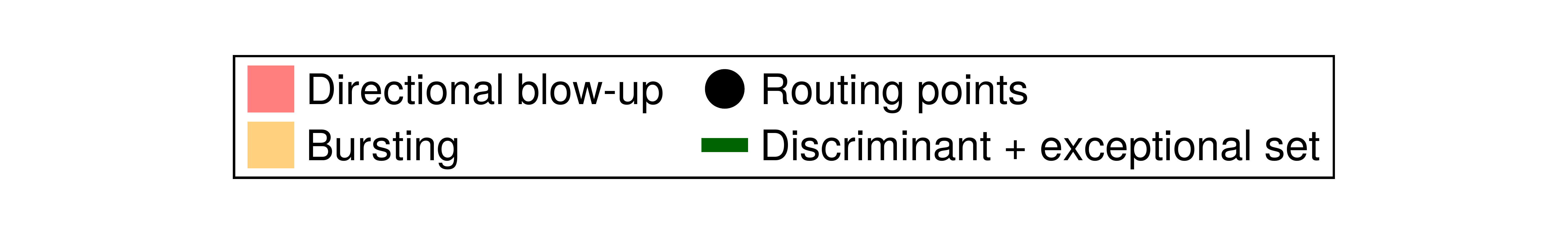}}
\vspace{-15pt}
\caption{Bursting landscapes for \cref{eq:knobloch}, over the imaginary coefficients $(B_I,C_I)$, at four choices of the fixed real coefficients. 
Panel (a) reproduces the geometry of ref.~\onlinecite[Fig.~1]{moehlis2000bursts}.}
\label{fig:knobloch}
\end{figure}

\section{Discussion and Future Directions}

Using the method of lines to discretize a polynomial partial differential equation (PDE) can result in an autonomous polynomial ODE of the form suitable for compactification-based blow-up analysis~\cite{matsue2017numerical}. One might hope that, for suitable PDEs, the behavior of the PDE is eventually well-approximated by the behavior of its method-of-lines discretizations~\cite{schiesser2012numerical}. Unfortunately, we have found that none of the approaches we tested produced a meaningful parameter decomposition. Still, we believe that it is an intriguing question for future work whether there are some PDEs, or non-polynomial ODEs, whose blow-up discriminant boundaries admit successive polynomial approximation using our method.

The possibility of saying more about non-directional blow-up in three or more variables is also worth exploring. This would likely require more analytic arguments, but would lead to a significantly richer parameter landscape.

Moreover, we have only identified blow-up here, and have said little about what can be done through or past it. Compactification is known to allow ``going past infinity'' in dynamical systems~\cite{kevrekedis2017infinity}, and it could be fruitful to combine this with our polynomial-focused techniques.

\section*{Supplementary Material}

The supplementary material contains the defining equations for the discriminant sets and the proof of~\cref{prop:decomp}.

\begin{acknowledgments}
E. G. and A. T. have been supported by the Defense Advanced Research Projects Agency (DARPA) through The Right Space (TRS) Disruption Opportunity
25 (DARPA-PA-24-04-07). The views, findings, and conclusions expressed in this paper are those of the authors and do not necessarily reflect the official policy or position of DARPA, the U.S. Department of Defense, or the U.S. Government. E.G. was also supported by NSF GRFP (DGE-2139899). A. T. has also been supported by the National Science Foundation CAREER grant DMS-2045646. The work of Y. G. K. has been partially supported by the U.S. Department of Energy and the US National Science Foundation.
\end{acknowledgments}

\section*{Author Declarations}

\subsection*{Author Contributions}
Emil Graf: Conceptualization, Methodology, Software, Formal analysis, Visualization, Writing. 
Yannis G. Kevrekidis: Conceptualization, Supervision, Writing. 
Alex Townsend: Conceptualization, Supervision, Writing.

\subsection*{Conflict of Interest}
The authors have no conflicts to disclose.

\section*{Data Availability Statement}
The code to reproduce all the figures in this paper is publicly available in a GitHub repository~\cite{EmilGithub}. 

\nocite{khalil2002nonlinear,teschl2012ordinary,krupa1995asymptotic}
\bibliography{references}

\newpage

\setcounter{equation}{0}
\renewcommand{\theequation}{S\arabic{equation}}

\setcounter{section}{0}
\renewcommand{\thesection}{S\arabic{section}}

\newtheorem*{propManual1}{Proposition \ref{prop:infconv}}
\newtheorem*{propManual2}{Proposition \ref{prop:decomp}}

\clearpage
\onecolumngrid

{\Large \textbf{Supplementary Material}}

\section{Introduction}
The goal of this supplement is to give the equations defining the discriminant sets used in the parameter-space decomposition, and to justify why the blow-up classification is constant on the resulting regions away from exceptional parameter values. Our starting point is the ordinary differential equation (ODE) system
\begin{equation} \label{eq:ODEs}
        \dot{\xx}=f(\xx,\lambda), \qquad \xx(t)\in\R^d,\quad \lambda\in\R^m,
\end{equation}
where $f$ is a polynomial of degree $n$ in $\xx$. If there exists a finite $t^*$ such that $\|\xx(t)\|\to\infty$ as $t\uparrow t^*,$
we say that the trajectory has a finite-time blow-up. We say that a system has finite-time blow-up if there is some finite initial condition $\xx(0)$ such that the trajectory starting from $\xx(0)$ has a finite-time blow-up. 

\section{Compactification of Trajectories}

We review the Poincar\'e compactification from the main text. We compactify the state space by embedding $\R^d \hookrightarrow S^{d} \subset \R^{d+1}$, the unit sphere in $\R^{d+1}$. The embedding is 
$$\xx \mapsto \left( \frac{\xx}{\sqrt{1+\|\xx\|^2}},\frac{1}{\sqrt{1+\|\xx\|^2}}\right) := (\zz,y).$$

Let $\kappa = \sqrt{1+\|\xx\|^2}$. We can derive the $\zz$ and $y$ dynamics as
\begin{equation*}
\begin{aligned}
    \frac{d \zz}{dt} &= \frac{d}{dt} (\xx/\kappa(\xx)) = \kappa^{-1} \frac{d \xx}{d t} - \kappa^{-2} \left\langle \nabla \kappa(\xx), \frac{d \xx}{dt} \right\rangle \xx \\
    &= \kappa^{-1} [ f(\xx,\lambda) - \kappa^{-1} \left\langle \zz, f(\xx,\lambda) \right\rangle \xx] \\
    &= \kappa^{-1} [ f(\kappa \zz,\lambda) - \langle \zz, f(\kappa \zz,\lambda) \rangle \zz], \\
    \frac{dy}{dt}& = -\frac{1}{2}(1+\|\xx\|^2)^{-3/2} \frac{d}{dt}\|
    \xx\|^2 = -y^3 \langle \xx, f(\xx,\lambda) \rangle.
\end{aligned}
\end{equation*}
Write
$
f(\xx,\lambda) = f_0(\xx,\lambda) + \cdots + f_n(\xx,\lambda),
$
where each $f_i$ consists of the degree $i$ homogeneous terms of $f$ in $\xx$. Define 
\begin{equation} \label{eq:ftilde}
\tilde{f}(\zz,\kappa,\lambda) \! = \! \kappa^{-n} f(\kappa \zz,\lambda)  \! = \! \kappa^{-n} f_0(\zz,\lambda) \! + \! \cdots \! + \! f_n(\zz,\lambda).
\end{equation}
Writing $\kappa = \kappa(T^{-1}(\zz,\lambda))$, where $T(\xx,\lambda) = \xx/\kappa$, we can rewrite the system as
\begin{equation} \label{eq:zsys}
    \frac{d \zz}{dt} = [\kappa(T^{-1}(\zz,\lambda))]^{n-1} [ \tilde{f}(\zz,\lambda) - \langle \nabla \kappa, \tilde{f}(\zz,\lambda) \rangle \zz].
\end{equation}
Then let
$
    \tau = \int_0^t (\kappa(\xx(s)))^{n-1} ds,
$
so that
$
    \frac{d\tau}{d t} = \kappa(\xx(t))^{n-1}.
$
Then we can rewrite the dynamics with respect to $\tau$ by
\begin{equation} \label{eq:tausys}
\begin{aligned}
    \frac{d \zz}{d\tau} &= \tilde{f}(\zz,\lambda) - \langle \nabla \kappa, \tilde{f}(\zz,\lambda) \rangle \zz := g(\zz,\lambda) \\
    \frac{d y}{d \tau} &= -y^{n+2} \langle \xx, f(\xx,\lambda) \rangle
\end{aligned}
\end{equation}
Every original trajectory that tends to infinity gives an equilibrium point at infinity in the transformed coordinates. We restate the following  proposition.
\begin{propManual1}[Ref.~\onlinecite{elias2006critical},~Prop.~2.5]
    Assume that a solution $\xx(t)$ has a maximal interval of existence $(a,b) \subset \R$, and that $\xx$ tends to infinity in the direction $\xx^*$. Then $\xx^*$ is an equilibrium point of \cref{eq:tausys} with $y = 0$.
\end{propManual1}
We use the integral
\begin{equation} \label{eq:tmax}
    t_{\max} = \int_{0}^{\infty} \frac{d \tau}{(\sqrt{1 + \|T^{-1}(\zz(\tau))\|^2})^{n-1}}
\end{equation}
to distinguish true blow-up trajectories. The trajectory blows up in finite time if and only if $t_{\max} < \infty$.

\section{Discriminant Boundaries in Parameter Space}
\label{sec:discbound}
Recall the five discriminant boundaries:
\begin{enumerate}
    \item ``The Algebraic Discriminant'': The creation or destruction of an equilibrium at infinity. 
    \item ``The Radial Discriminant'': The change of an equilibrium at infinity between attracting and repelling.
    \item ``The Stability Discriminant'': The change of an eigenvalue of the Jacobian of $f$ at an equilibrium from positive real part to negative real part.
    \item ``The Degree Discriminant'': The change in degree of $f$.
    \item ``The Positive Discriminant'': (optional) The creation or destruction of positive equilibria at infinity.
\end{enumerate}
In the language of bifurcation theory, the algebraic discriminant captures steady-state bifurcations, and the stability discriminant captures oscillatory or Andronov-Hopf bifurcations.
We call each of these a discriminant as they are similar to the classical algebraic geometry discriminant. We now explain all the discriminants in detail.
\subsection{The Algebraic Discriminant}

At infinity ($y=0$), $g=0$ is equivalent to
$$
\tilde{g}(\zz,\lambda) = f_n(\zz,\lambda) - \langle \zz, f_n(\zz,\lambda) \rangle \zz = 0.
$$
In addition, we enforce the scale so that the equilibria at infinity live on $\|\zz\|^2=1$.
The nonzero real roots of this polynomial determine the equilibria of the system at infinity. The real roots can only change at parameter values where $\tilde{g}(\zz,\lambda)$ and  $\det(J_{\zz} \tilde{g}(\zz,\lambda))$ have a common zero, where $J_{\zz}$ denotes the Jacobian with respect to $\zz$. This defines an algebraic set in the $\zz$ and $\lambda$ variables. The algebraic discriminant locus $X_A$ is the projection of this set onto the $\lambda$ variables, and the algebraic discriminant $\Delta_A$ is the set of its defining equations, so that $X_A = \{\lambda \in \R^m : h(\lambda) = 0 \, \forall \, h \in \Delta_A\}$.

\begin{ex} \label{ex:alg}
Consider the system
\begin{equation*}
\begin{aligned}
    \dot{x}_1 &= x_1^2 \\
    \dot{x}_2 &= \lambda x_1^2 - x_2^2
\end{aligned}
\end{equation*}
At infinity, the tangential vector field is
\begin{equation}
\tilde{g}(\mathbf{z}) = \begin{bmatrix} 
    z_1^2 - (z_1^3 + \lambda z_1^2 z_2 - z_2^3)z_1 \\
    \lambda z_1^2 - z_2^2 - (z_1^3 + \lambda z_1^2 z_2 - z_2^3)z_2
\end{bmatrix} = \begin{bmatrix}
    z_1 z_2 (z_1 z_2 - \lambda z_1^2 + z_2^2) \\
    z_1^2 (\lambda z_1^2 - z_1 z_2 - z_2^2)
\end{bmatrix},
\end{equation}
subject to $z_1^2 + z_2^2 = 1$. The points $(0, \pm 1)$ are equilibria for all values of $\lambda$. If $z_1 \neq 0$, then $\tilde{g} = 0$ is equivalent to the quadratic condition:
\begin{equation}
\lambda z_1^2 - z_1 z_2 - z_2^2 = 0.
\end{equation}
Dividing by $z_1^2$ and letting $u = z_2/z_1$, we obtain the characteristic equation $u^2 + u - \lambda = 0$. The roots for $u$ are given by 
\begin{equation}
u = \frac{-1 \pm \sqrt{1 + 4\lambda}}{2}.
\end{equation}
Real solutions for $u$ (and thus additional equilibria on the circle) exist if and only if $1 + 4\lambda \ge 0$. The algebraic discriminant locus is the point where these roots collide and disappear, which is $X_A = \{-1/4\}$. Consequently, there are six real equilibria at infinity for $\lambda > -1/4$ (the two static roots and the four roots corresponding to the two values of $u$) and only two real equilibria for $\lambda < -1/4$.
\end{ex}

\subsection{The Radial Discriminant} 

Recall that as $\xx \to \infty$, $y \to 0$. Then $f(\xx,\lambda) = y^{-n} f_n(\zz,\lambda) + \mathcal{O}(y^{-(n-1)})$, so
    $$
    \frac{dy}{d \tau} = -y^{n+2} \langle y^{-1} \zz,y^{-n} f_n(\zz,\lambda) \rangle + \mathcal{O}(y^2).
    $$
    We let $\mu_\perp$, the radial eigenvalue, be the derivative of the $y$-dynamics evaluated at $y=0$, that is
    $$
    \mu_{\perp} = \frac{\partial}{\partial y} \left. \left( \frac{dy}{d \tau} \right) \right\vert_{y=0} = - \langle \zz, f_n(\zz,\lambda) \rangle
    $$
    We define the radial discriminant locus $X_R$ to be the projection of the vanishing set of $\tilde{g}(\zz,\lambda)$ and $\langle \zz, f_n(\zz,\lambda) \rangle$ onto the $\lambda$ coordinates, with $\Delta_R$ its defining polynomials. 
    
    A strict definition of the algebraic discriminant would project the radial direction out, so that the algebraic discriminant governs only the tangential behavior on the boundary of the sphere. It is far simpler computationally to skip this step, in which case the radial component is included in the full Jacobian of the system, so $X_R \subset X_A$ and, if both are a single polynomial, $\Delta_R \mid \Delta_A$. In our experiments we compute $\Delta_R$ for additional information about the system but we do not use it to decompose the parameter space because it adds no new regions beyond those of $\Delta_A$. 

\begin{ex} \label{ex:rad}
Consider a univariate system ($d=1$) with parameter $\lambda \in \R$ given by $\dot{x} = \lambda x^2$. There are two total points at infinity, given by $z=\pm 1,y=0$. For univariate systems, both points at infinity are equilibria, because there are no tangential directions at infinity, and by definition the boundary is an invariant manifold (in the radial direction). Thus, both points are equilibria for all $\lambda$, and the only question is whether they are attracting. We have
$$\mu_{\perp}(1,0) = -\lambda, \qquad \mu_{\perp}(-1,0) = \lambda,$$
so the radial discriminant locus is $X_R = \{0\}$. For every $\lambda\neq0$, the system has finite-time blow-up from some initial condition, but the blow-up direction changes as $\lambda$ crosses zero. At $\lambda=0$, the degree drops and there is no blow-up. If we restrict only to positive blow-up, then the discriminant divides regions with distinct blow-up behavior.
\end{ex}

\subsection{The Stability Discriminant}

The eigenvalues of $J_{\zz} \tilde{g}$ can affect the convergence of the integral in \cref{eq:tmax}, and they also affect qualitative behavior of possible blow-ups. The stability discriminant distinguishes between regions where an equilibrium has eigenvalues with positive versus negative real parts. We consider the algebraic set consisting of the set of all $(\zz,\lambda)$ pairs such that there is an equilibrium with a purely imaginary eigenvalue. This condition can be detected using the Routh--Hurwitz conditions, which tell us when a matrix has an eigenvalue with zero real part. Specifically, let
$$
\textrm{det} (\Lambda I_n - J_{\zz} \tilde{g}(\zz,\lambda)) = c_n(\zz,\lambda)\Lambda^n + \cdots + c_0(\zz,\lambda),
$$
where each $c_i$ is polynomial in $\zz,\lambda$. Then the stability discriminant locus $X_S$ is given by the projection of the union of the vanishing sets of $c_i,i=0,\ldots,n$ onto the $\lambda$ coordinates. Similarly, we let the stability discriminant $\Delta_S$ be its set of defining polynomials, so that $X_S = \{\lambda \in \R^m : \Delta_S(\lambda) = 0 \}$. Notice that the stability discriminant locus contains the algebraic discriminant locus, that is $X_A \subset X_S$.

However, we will see that off of the stability discriminant, the blow-up behavior is only controlled by the radial eigenvalue, not by the eigenvalues of $J_{\zz}$. Therefore, this discriminant forms an exceptional set that must be avoided for our analysis to hold, but it does not split the parameter space into regions with distinct blow-up behaviors. As $X_S$ is a lower-dimensional algebraic subset, it is measure zero, so generic parameter choices avoid it.

\begin{ex} \label{ex:stab}
The system in this example cannot experience finite-time blow-up because it is linear, but it can have grow-up trajectories; it is designed to illustrate the stability boundary. To observe a stability boundary distinct from the algebraic boundary (where $X_S \setminus X_A \neq \emptyset$), the sphere at infinity must be at least two-dimensional. Consider a three-dimensional linear system ($d=3, n=1$) parameterized by $\lambda \in \R$: $\dot{x}_1 = x_1$, $\dot{x}_2 = \lambda x_2 - x_3$, and $\dot{x}_3 = x_2 + \lambda x_3$. 

At infinity, we have
$$
\tilde{g}(\zz,\lambda) = \begin{bmatrix}
    z_1 (1-r)\\
    \lambda z_2 - z_3 - rz_2 \\
    z_2 + \lambda z_3 - rz_3
\end{bmatrix} = 0,
$$
with
$
r = z_1^2 + \lambda (z_2^2 + z_3^2),
$
and
\[
J\tilde g(\mathbf z)=
\begin{bmatrix}
1 - r - 2z_1^2 & -2\lambda z_1 z_2 & -2\lambda z_1 z_3 \\
-2 z_1 z_2 & \lambda - r - 2\lambda z_2^2 & -1 - 2\lambda z_2 z_3 \\
-2 z_1 z_3 & 1 - 2\lambda z_2 z_3 & \lambda - r - 2\lambda z_3^2
\end{bmatrix}.
\]
The points $(\pm 1, 0, 0)^\top$ on $S^2$ are the only equilibria at infinity. At either equilibrium, the eigenvalues of $J$ are $-2, \lambda-1+i,\lambda-1-i$.
The equilibrium undergoes a Hopf bifurcation at infinity when $\lambda = 1$, crossing the imaginary axis without passing through zero. Thus the stability discriminant locus is $X_S = \{1\}$, while the algebraic locus $X_A$ remains empty.
\end{ex}

\subsection{The Degree Discriminant}

The convergence of the integral in \cref{eq:tmax} is also affected by the degree of $f(\zz,\lambda)$ in $\zz$ for particular choices of $\lambda$. If the degree drops below $n$, then the theory breaks down, as the coordinate system is built for $f(\zz,\lambda)$ to remain maximal degree $n$. Therefore, there is an algebraic exceptional set, which we call the degree discriminant locus, and denote $X_D$, with defining set $\Delta_D$, where $\Delta_D$ can be written explicitly as the set of all coefficient functions (of $\lambda$) for the degree $n$ homogeneous terms of $f_n$. Equivalently, $X_D = \{ \lambda \in \R^m : f_n(\zz,\lambda) = 0 \, \forall \zz \in \R^d \}$. Note that the degree discriminant is actually contained within the algebraic discriminant.

As with the stability discriminant, the degree discriminant does not split the parameter space into regions where the blow-up behavior may change. The degree is constant everywhere outside $X_D$. As $X_D$ is a lower-dimensional algebraic subset, it is measure $0$, so generic parameter choices avoid it. If desired, it is also possible to specialize to parameters solely inside $X_D$, and then re-examine other discriminants.

\begin{ex} \label{ex:deg}
Consider a univariate system ($d=1$) parameterized by $\lambda \in \R$ given by $\dot{x} = \lambda x^2 - x$. The intended degree of the system is $n=2$, with the leading homogeneous term being $f_2(x) = \lambda x^2$. 

At $\lambda = 0$, the function $f_2(x, 0)$ vanishes for all $x \in \R$. The physical system drops to degree $n=1$ ($\dot{x} = -x$), meaning the coordinate transformation scaled by $y^2$ is no longer valid, and trajectories no longer reach the sphere at infinity in finite time. The degree discriminant locus is $X_D = \{0\}$.
\end{ex}

\subsection{The Positive Discriminant}

For many physical systems, only some values of $\xx$ are meaningful. It is common for only positive $\xx$ coordinates to represent physical positions of the system, for example if the $\xx$ variables are populations, energies, etc. In this case, we introduce a variant of the algebraic discriminant, called the positive discriminant. 
We still have that the equilibria are the zeros of the highest degree homogeneous part
$$
\tilde{g}(\zz,\lambda) = f_n(\zz,\lambda) - \langle \zz, f_n(\zz,\lambda) \rangle \zz = 0.
$$
The positive discriminant locus $X_P$ is the projection of the vanishing set of $\tilde{g}$ and at least one of the coordinate functions $z_i$, with defining set $\Delta_P$.

Consider the system in \cref{ex:rad}. The positive discriminant is empty, but if we only consider the positive blow-ups, then the radial discriminant $\lambda = 0$ splits the parameter space between regions with and without blow-up.

\begin{ex} \label{ex:pos}
    Consider the system 
\[
\dot x_1 = x_1^2 - \lambda x_2^2, \qquad
\dot x_2 = x_2^2 - (1+\lambda)x_1^2.
\]
The leading homogeneous part is
\[
f_2(x) =
\begin{bmatrix}
x_1^2 - \lambda x_2^2 \\
x_2^2 - (1+\lambda)x_1^2
\end{bmatrix}.
\]
The inner product is
\[
\langle x, f_2(x)\rangle
= x_1^3 - \lambda x_1 x_2^2 + x_2^3 - (1+\lambda)x_1^2 x_2.
\]
Thus the compactified vector field
\[
\tilde g(x) = f_2(x) - \langle x, f_2(x)\rangle x
\]
has components
\[
\tilde g_1 =
x_1^2 - \lambda x_2^2
- x_1^4 + \lambda x_1^2 x_2^2
- x_1 x_2^3 + (1+\lambda)x_1^3 x_2,
\]
\[
\tilde g_2 =
x_2^2 - (1+\lambda)x_1^2
- x_1^3 x_2 + \lambda x_1 x_2^3
- x_2^4 + (1+\lambda)x_1^2 x_2^2.
\]
Evaluating at points where the coordinates vanish gives
\begin{equation*}
\begin{aligned}
    \tilde g(0,\pm1) &= (-\lambda, 0), \\
    \tilde g(\pm1,0) &= (0, -(\lambda+1)), \\
\end{aligned}
\end{equation*}
so the positive discriminant locus is $X_P = \{0,-1\}$.

\end{ex}

\subsection{Demonstrating We Have the Entire Discriminant}
\label{subsec:decompproof}

As in the main text \cref{sec:disc}, as long as $\lambda \in \R^m \setminus (X_S \cup X_D)$, in the cells defined by $X_A \cup X_R$, the following are constant:
\begin{enumerate}
    \item The number of equilibria at infinity and the sign of the radial eigenvalue $\mu_\perp$.
    \item The degree of $f$, which is always equal to $n$ on the complement of the degree discriminant.
\end{enumerate}
We prove \cref{prop:decomp}, which proves that our algebraic decomposition of the parameter space actually corresponds to the claimed blow-up behavior.

\begin{propManual2}
    Suppose $\lambda$ is fixed at a parameter value for which the degree of $f(\xx,\lambda)$ in $\xx$ is $n$, and all equilibria at infinity are hyperbolic. Suppose $n > 1$: If there is an equilibrium at infinity with $\mu_\perp < 0$, then there is an initial condition for which the system $\dot{\xx} = f(\xx,\lambda)$ blows up in finite time. Suppose $d \leq 2 $. If there exists at least one equilibrium at infinity, but there are no equilibria with $\mu_\perp < 0$, then the system does not have a finite-time blow-up for any initial condition.
    If $n = 1$, then there is no finite-time blow-up for any initial condition.
\end{propManual2}

\begin{proof}
    Previous work, as outlined in the main text \cref{sec:background}, reduces this to the question of whether there is a trajectory in the compactified space that approaches an equilibrium at infinity, such that
    $$
    t_{\max} = \int_{0}^{\infty} \frac{d \tau}{(\sqrt{1 + \|T^{-1}(\zz(\tau))\|^2})^{n-1}}
    $$
    converges. 
    The convergence is governed by the behavior of the $y$-component of the trajectory. Recall that the embedding $\R^d \hookrightarrow S^{d}$ is defined by the map $\xx \mapsto (\zz, y)$, where
    \begin{equation}
        \zz = T(\xx,\lambda) = \frac{\xx}{\sqrt{1 + \|\xx\|^2}}, \quad y = \frac{1}{\sqrt{1 + \|\xx\|^2}},
    \end{equation}
    with $T(\xx,\lambda) = \xx/\kappa$.
    Consequently, the integrand in \cref{eq:tmax} is $y(\tau)^{n-1}$, so we can write
    \begin{equation}
        t_{\max} = \int_{0}^{\infty} y(\tau)^{n-1} d\tau.
    \end{equation}
    Near an equilibrium, the behavior is governed by the linearization
    \begin{equation} \label{eq:linearization}
    \frac{d}{d\tau} \begin{pmatrix} \delta \mathbf{z} \\ y \end{pmatrix} = \begin{pmatrix} J_{\mathbf{z}} & \mathbf{w} \\ \mathbf{0}^\top & \mu_{\perp} \end{pmatrix} \begin{pmatrix} \delta \mathbf{z} \\ y \end{pmatrix} + \mathcal{O}(|(\delta \mathbf{z}, y)|^2),
    \end{equation}
    for some vector $\ww \in \R^d$.
    Since we avoid the exceptional stability discriminant $X_S$, all equilibria at infinity are hyperbolic by assumption. If an equilibrium has $\mu_\perp < 0$, it possesses a stable manifold that intersects the physical domain $y > 0$, i.e., there is a physical initial condition that converges to the equilibrium.
    To determine the convergence of $t_{\max}$ for a trajectory approaching such an equilibrium, we isolate the $y$-dynamics:
    $$
    \frac{dy}{d\tau} = \mu_\perp y + h(\delta \zz, y) y,
    $$
    where $h(\delta \zz, y) = \mathcal{O}(|(\delta \zz, y)|)$ represents the higher order terms. Because the trajectory converges to the equilibrium $(0,0)$, for any $\epsilon > 0$, there exists a time $\tau_0 > 0$ such that for all $\tau \ge \tau_0$, we have $|h(\delta \zz(\tau), y(\tau))| < \epsilon$. 
    Using the comparison lemma for ordinary differential equations (ref.~\onlinecite[lemma 3.2]{khalil2002nonlinear}), we can bound the derivative of $y$:
    $$
    (\mu_\perp - \epsilon) y \le \frac{dy}{d\tau} \le (\mu_\perp + \epsilon) y, \quad \forall \tau \ge \tau_0.
    $$
    Integrating this differential inequality from $\tau_0$ to $\tau$ yields exponential bounds:
    $$
    y(\tau_0) e^{(\mu_\perp - \epsilon)(\tau - \tau_0)} \le y(\tau) \le y(\tau_0) e^{(\mu_\perp + \epsilon)(\tau - \tau_0)}.
    $$
    Now we evaluate the integral for $t_{\max}$. Since the trajectory takes finite physical time to reach $\tau_0$, the total time is finite if and only if the tail integral converges:
    $$
    \int_{\tau_0}^{\infty} y(\tau)^{n-1} d\tau.
    $$
    Substituting our upper and lower bounds gives:
    $$
    C_1 \int_{\tau_0}^{\infty} e^{(\mu_\perp - \epsilon)(n-1)\tau} d\tau \le \int_{\tau_0}^{\infty} y(\tau)^{n-1} d\tau \le C_2 \int_{\tau_0}^{\infty} e^{(\mu_\perp + \epsilon)(n-1)\tau} d\tau,
    $$
    where $C_1, C_2 > 0$ are constants depending on $y(\tau_0)$, $\epsilon$, and $\tau_0$.
    
    Assume $n > 1$. Because $\mu_\perp < 0$, we can choose $\epsilon$ sufficiently small such that $\mu_\perp + \epsilon < 0$. Therefore, the exponent $(\mu_\perp + \epsilon)(n-1)$ is negative, and the upper bound integral converges to a finite value. This proves that $t_{\max} < \infty$, so the system has a finite-time blow-up. Conversely, if $n = 1$, the integral simplifies to $\int_{\tau_0}^{\infty} 1 d\tau$, which diverges to infinity regardless of the value of $\mu_\perp$. Hence, if $n=1$ the system cannot blow up in finite time.
    
    Finally, suppose $d \leq 2$, that there is at least one equilibrium at infinity ($y=0$), but there are no equilibria at infinity with $\mu_\perp < 0$. Note that the Poincar\'e-Bendixson theorem implies that the only possible limiting behavior on the circle $S^1$ or the trivial point at infinity (if $d=1$) is convergence to a fixed point or a finite set of fixed points connected by homoclinic and heteroclinic orbits (a graphic) (ref.~\onlinecite[Thm. 7.16]{teschl2012ordinary}). None of the fixed points at infinity are attracting from the physical space, so, as it takes any trajectory an infinite amount of fictitious time $\tau$ to reach infinity, it is impossible for a trajectory to converge to a fixed point at infinity. 
    
    Suppose $d=2$ and the limit set is a graphic. We use an argument similar to that of ref.~\onlinecite{krupa1995asymptotic} to demonstrate that the amount of time that the trajectory spends in small neighborhoods of the equilibria per cycle diverges to infinity. For small enough neighborhoods of each equilibrium the trajectory is governed by the linearization in \cref{eq:linearization}. Suppose a small neighborhood of an equilibrium in coordinates centered at the equilibrium, given by $(z,y) \in [-\delta,\delta] \times [0,\delta]$. The linearization pulls the trajectory towards $z=0$, so it must enter the box at some $y = y_{\textrm{in}}$ and leave at $y = \delta$. By solving $\delta = y_{\textrm{in}} e^{\mu_{\perp} T}$, we see that the transit time across the neighborhood is 
    $$
    T = \frac{1}{\mu_{\perp}} \log \left( \frac{\delta}{y_{\textrm{in}}} \right).
    $$
    As the trajectory approaches the boundary, $y_{\textrm{in}} \to 0$ but $\delta$ stays fixed, so the transit time approaches infinity. Between these neighborhoods, for large enough $\tau$, the trajectory does not pass through any equilibria, and so there is a finite lower bound on the speed $\|g(\zz(\tau))\|$. Therefore, the trajectory spends a finite amount of time outside the small neighborhoods of the equilibria. In each neighborhood, the velocity in the radial direction is always back into the finite region with $y > 0$, which contradicts the graphic being the limit set of the system. Therefore, blow-up is impossible.
\end{proof}

This shows that if we calculate the discriminants $\Delta_A$ and $\Delta_R$ that define the algebraic and radial discriminant sets, then they correctly separate regions with well-defined blow-up behavior. If only positive-orthant behavior is of interest, then we include $\Delta_P$ in the decomposition. If additional stability information is desired, we also use $\Delta_S$. Therefore, we have proved that the regions of the figures in the main text are correctly classified.

\end{document}